\documentclass[11pt]{article}
\usepackage{amsfonts}
\usepackage{mathrsfs}
\usepackage[dvips]{graphics}
\usepackage[dvips]{color}
\usepackage{amsthm}
\usepackage[]{amsmath}
\usepackage{CJK}
\usepackage{indentfirst}
\newtheorem{theorem}{Theorem}[section]
\newtheorem{definition}[theorem]{Definition}
\newtheorem{property}[theorem]{Property}
\newtheorem{lemma}[theorem]{Lemma}

\newcommand{\F}{\mathbb{F}}
\newcommand{\Ker}{\text{\rm Ker}}
\newcommand{\Fib}{\text{\rm Fib}}
\newcommand{\w}{\omega}

\newcommand{\of}{\mathcal{D}}
\newcommand{\PP}{{\mathcal{P}}}

\begin{document}
%\begin{CJK*}{GBK}{song}
%\CJKindent

\centerline{\textbf{\LARGE{The Number of Fractional Powers in the Fibonacci Word}}}

\vspace{0.2cm}

\centerline{Huang Yuke\footnote[1]{School of Science, Beijing University of Posts and Telecommunications, Beijing, 100876, P. R. China. E-mail address: hyk@bupt.edu.cn, hyg03ster@163.com (Corresponding author).}
~~Wen Zhiying\footnote[2]{Department of Mathematical Sciences, Tsinghua University, Beijing, 100084, P. R. China. E-mail address: wenzy@tsinghua.edu.cn.}}

\vspace{1cm}

\centerline{\large{ABSTRACT}}

\vspace{0.2cm}

The Fibonacci word is
the fixed point beginning with the letter $a$ of morphism $\sigma(a)=ab$, $\sigma(b)=a$
defined over the alphabet $\{a,b\}$.
In this paper, we get explicit expression of
the number of distinct fractional powers in each factor of the Fibonacci word.

\vspace{0.2cm}


\vspace{0.2cm}

\noindent\textbf{2010 MR Subject Classification}~~~~68R15

\section{Introduction}

Let $\mathcal{A}$ be a finite set called alphabet.
Let $\w=x_1x_2\cdots x_N$ be a finite word over $\mathcal{A}$.
We denote by $|\w|$ (resp. $|\w|_\alpha$, $\#\w$) the number of letters (resp. letter $\alpha$, distinct letters) in $\w$.
For $1\leq i\leq j\leq N$,
we define $\w[i,j]=x_ix_{i+1}\cdots x_{j-1}x_j$ or $[x_i,x_{i+1},\ldots,x_{j-1},x_j]$£¬
by convention, $\w[i]=\w[i,i]=x_i$ and $\w[i,i-1]=\varepsilon$ (empty word).
We call $\w[i,j]$ (resp. $\w[1,j]$, $\w[i,N]$) a factor (resp. prefix, suffix) of $\w$, denoted by $\w[i,j]\prec\w$ (resp. $\w[1,j]\triangleleft\w$, $\w[i,N]\triangleright\w$).
For any $i\geq1$ and $i+n-1\leq N$, we define $\w[i;n]=\w[i,i+n-1]$.
%The $k$-th conjugation of $\w$ is
%$C_{k}(\w)=x_{k+1}x_{k+2}\cdots x_Nx_1x_{2}\cdots x_{k}$ for $1\leq k\leq N$.
The position of factor $\w[i,j]$ in $\w$ is defined by $i$.
An infinite word is said to be recurrent if every factor occurs infinitely often \cite{AS2003}.
Let $\w$ be a factor of a recurrent infinite word.
We denote the position of the $p$-th occurrence of $\w$ by $(\w)_p$. When $|\w|=1$, we denote $(\w)_p$ by $\w_p$ for short.
Let $\w$ and $\nu$ be two words. $\w\cdot\nu$ is the concatenation of $\w$ and $\nu$.
%Let $M$ be a matrix. $M_{i,j}$ means the element in row $i$ and column $j$.
%If $M$ is one-dimension, $M[i]$ means its $i$-th element.
Let $S$ be a set. We denote $\# S$ the number of elements in $S$.
For a real number $r$, the floor function $\lfloor r\rfloor$ is the greatest integer less than or equal to $r$, and the ceiling function $\lceil r\rceil$ is the least integer greater than or equal to $r$.

%\smallskip

The Fibonacci morphism $\sigma$ over alphabet $\{a,b\}$ is a substitution defined by $\sigma(a)=ab$ and $\sigma(b)=a$.
The Fibonacci word $\F$ is defined to be
the fixed point beginning with the letter $a$ of the Fibonacci morphism.
It is a recurrent infinite word \cite{AS2003}.
Define $F_m=\sigma^m(a)$ for $m\geq0$, by convention,
$F_{-1}=b$ and $F_{-2}=\varepsilon$.
The $m$-th Fibonacci number $f_m$ is equal to $|F_m|$.
As a classical example of sequences over the binary alphabet,
the Fibonacci word
$\F$ has many remarkable properties. We refer to Lothaire \cite{L1983,L2002}, Allouche-Shallit \cite{AS2003} and Berstel \cite{B1996}.

Let $R_a$ and $R_b$ be two words over alphabet $\mathcal{A}$.
The Fibonacci word over alphabet $\{R_a,R_b\}$ is denoted by $\F(R_a,R_b)$.
In this paper, we always assume that there exists an integer $m\geq0$ such that $|R_a|=f_m$ and $|R_b|=f_{m-1}$.
In particular, when $\mathcal{A}=\{a,b\}$, $R_a=a$ and $R_b=b$, we have $\F(a,b)=\F$.
We also regard $\F(R_a,R_b)$ still a sequence over alphabet $\mathcal{A}$ if no confusion happens. In this case we use $\F(R_a,R_b)[i]$ as
the $i$-th letter in the sequence. For instance, let $\mathcal{A}=\{A,B,C,D,E\}$, $R_a=ABC$ and $R_b=DE$, then
$\F(R_a,R_b)=R_aR_bR_aR_aR_b\cdots=ABC\cdot DE\cdot ABC\cdot ABC\cdot DE\cdots$.
And $\F(R_a,R_b)[5]=E$.

By Zeckendorf numeration system \cite{Z1972}, every positive integer $n$ can be written uniquely as $n=\sum_{j\geq0}a_jf_j$ with $a_j\in\{0,1\}$ and $a_{j+1}\times a_j=0$ for $j\geq0$. If $a_k=1$ and $a_j=0$ for all $j>k$,
we call $\Fib(n)=a_ka_{k-1}\ldots a_1a_0$
the canonical Fibonacci representation of $n$.
Let $a_ma_{m-1}\cdots a_1a_0$ be a word
over alphabet $\{0,1\}$, by \cite{DMSS2015-1}, the word can be written as a number $[a_ma_{m-1}\cdots a_1a_0]_F=\sum^m_{j\geq0}a_jf_j$,
even if $a_m=0$ or $a_{j+1}\times a_j\neq0$ for some $j\geq0$.

%\smallskip

Let $u=x_1x_2\cdots x_N$ be a finite word over alphabet $\mathbb{Z}$ (the set of integers).
We denote $\Sigma u=x_1+x_2+\cdots +x_N$. If an infinite word $\rho$ satisfies that $\rho[i]=0$ for all $i\geq1$, we call it a zero sequence. Moreover,
$[\underbrace{n\longrightarrow}_{N}]:=[\underbrace{n,n,\ldots,n}_{N}]$,
$[\underbrace{n\nearrow}_{N}]:=[\underbrace{n,n+1,\ldots,n+N-1}_{N}]$ and
$[\underbrace{n\searrow}_{N}]:=[\underbrace{n,n-1,\ldots,n-N+1}_{N}]$.

%For a real number $r$, the floor function $\lfloor r\rfloor$ is the greatest integer less than or equal to $r$. Similarly, the ceiling function $\lceil r\rceil$ is the least integer greater than or equal to $r$.

The \emph{fractional power} is a topic dealing with repetitions in words.
We say a (finite or infinite) word $\w$ contains a $r$-power (real $r>1$) if $\w$ has a factor of the form $x^{\lfloor r\rfloor}x'$ where $x'$ is a prefix of $x$ and $|x^{\lfloor r\rfloor}x'|\geq r|x|$ see \cite{AS2003}. In this case, we call $x^{\lfloor r\rfloor}x'$ a $r$-power with \emph{size} $|x|$.
For instance, taking $x=ab$, then $\F[4,8]=ababa$ is a $\frac{5}{2}$-power of size $|ab|=2$ in $\F$.
Obviously the notion $r$-power is a generalization of square (2-power) and cube (3-power).

The study of power of a word has a long history. There are many significant contributions, for example \cite{DL2003,DMSS2016-3,FS1999,FS2014,G2006,IMS1997,MS2014,S2010,TW2007,WW1994}. %JP2002,
In particular, Iliopoulos-Moore-Smyth \cite{IMS1997} computed the positions of all squares in $\F$.
Fraenkel-Simpson \cite{FS1999,FS2014} obtained the number of %distinct (resp. repeated)
squares in $F_m=\sigma^m(a)$.
Du-Mousavi-Schaeffer-Shallit \cite{DMSS2016-3} obtained the numbers of repeated squares and cubes
in $\F[1,n]$ for all $n\geq1$. All their numerations start from the first letter of the sequence.

In this paper, we count the number of distinct $r$-powers in
$\mathbb{F}[i;n]$ for all $i,n\geq1$ and $r\geq2$, denoted by $\mathrm{D}(r,i,n)$.
%and $\mathrm{R}(r,i,n)$ respectively
Our numeration can start from any letter of the sequence, comparing starting from the first letter, there are some difficulties. To overcome these difficulties,
we develop some new techniques by using the structure of return word sequences introduced in \cite{HW2015-1,HW2015-2,HW2016-3}.
To get the general numeration of distinct $r$-powers, using again the structure of return word sequences, we establish a simple correspondent relation between the positions of squares and $r$-powers.

Our method is based on \emph{factor spectrum}, \emph{return word sequences} and \emph{kernel words}, established in Huang-Wen \cite{HW2015-1,HW2015-2,HW2016-3}.
This method can be used to study some other topics in combinatorics on words. For instance, using similar method, we can count the number of palindromes
in $\rho[i;n]$. Here $\rho$ is a substitution sequence and $i,n\geq1$.

We give our main idea and some simple examples in Section 2.
Then we list main results in Section 3. Sections 4, 5 and 7 will be devoted to the details of the proofs. We give some properties of function $D(2,i,n)$ in Section 6.

\section{Main idea and some simple examples}

We first recall \emph{factor spectrum} introduced in Huang-Wen \cite{HW2015-1,HW2015-2}.

Let $\rho$ be an infinite word and $\Omega$ be the set of all factors in $\rho$. A factor $\w\in\Omega$ may occur many times (even infinity many times) at different
positions in $\rho$. We note $(\w,j)$ the factor $\w$ which occurs at the position $j$. As we have mentioned before, any factor of the Fibonacci word will appear infinity many times.
Let $\cal P$ be a property. Our aim is to find out all factors which are located at different positions in $\rho$ satisfying the property.

\begin{definition}[Factor Spectrum]
Let $\cal P$ be a property, the factor spectrum with respect to $\PP$ is defined by
$$\mathcal{S}_{\PP}=\{(\w,j)\mid \rho[j;|w|]=\w\text{ and }(\w,j)\text{ satisfies the property }\PP\}.$$
Notice that $\mathcal{S}_{\PP}$ %$\subset \Omega\times\mathbb{Z}^{+}$
is a subset of $\Omega\times\mathbb{Z}^{+}$.
%in fact it is the set of $(\w,j)$ satisfying the property $\cal P$.
And it maybe a finite or infinite set according to different property $\PP$.
\end{definition}

We notice that factor spectrum with respect to $\PP$ can be regarded a function of two variables, factor variable $\w$ and position variable $j$.
%
%\bigskip
Factor spectrum plays an important role in the paper.
By factor spectrum, we get not only all positions of some fixed squares, but also all positions starting a square.
Let us start from a simple question.

\medskip

\textbf{Q1. How many repeated factor $aa$ in $\F[i;n]$ for $i,n\geq1$?}

Let us consider factor spectrum
$\mathcal{S}_{\PP_1}=\{(aa,j)\mid \F[j;2]=aa\}$.
It shows all positions of factor $aa$ in $\F$.
A powerful tool to determine $\mathcal{S}_{\PP_1}$ is the structure of the \emph{return word sequence} which is introduced and studied in Huang-Wen \cite{HW2015-1,HW2015-2},
see Property \ref{G} and Figure \ref{Fig:7} below.
The definition of \emph{return word} is from Durand \cite{D1998}.
Let $\w$ be a factor of $\F$.
The \emph{$p$-th return word} of $\w$ is $R_{p}(\w):=\F[(\w)_p,(\w)_{p+1}-1]$ for $p\geq1$.
The sequence $\{R_p(\w)\}_{p\geq1}$ is called the return word sequence of the factor $\w$.

\begin{property}[Theorem 2.11 in \cite{HW2015-1}]\label{G}\
For each factor $\w\prec\F$,
the return word sequence $\{R_p(\w)\}_{p\geq1}$ is itself still a Fibonacci word over the alphabet $\{R_1(\w),R_2(\w)\}$.
\end{property}

\begin{figure}[!ht]
\centering
\setlength{\unitlength}{0.83mm}
\begin{picture}(180,10)%75
\put(10,6){$\F[1,31]$}
\put(25,6){$=$}
\put(30,6){${\begin{array}{*{32}{p{0.3cm}}}
a&b&\textbf{a}&\textbf{a}&b&a&b&\textbf{a}&\textbf{a}&b&\textbf{a}&\textbf{a}
&b&a&b&\textbf{a}&\textbf{a}&b&a&b&\textbf{a}&\textbf{a}
&b&\textbf{a}&\textbf{a}&b&a&b&\textbf{a}&\textbf{a}&b\end{array}}$}
\put(30,4){\line(1,0){8}}
\put(38,5){\line(1,0){20}}
\put(58,4){\line(1,0){13}}
\put(71,5){\line(1,0){20}}
\put(91,4){\line(1,0){20}}
\put(111,5){\line(1,0){13}}
\put(124,4){\line(1,0){20}}
\put(30,0){$R_0$}
\put(38,0){A}
\put(59,0){B}
\put(71,0){A}
\put(91,0){A}
\put(112,0){B}
\put(124,0){A}
\end{picture}
\caption{The return word sequence $\{R_p(aa)\}_{p\geq1}$ is a Fibonacci word over the alphabet $\{A,B\}:=\{R_1(aa),R_2(aa)\}=\{aabab,aab\}$ with prefix $R_0=ab$.\label{Fig:7}}
\end{figure}

Obviously, the return word sequence gives a decomposition of $\F$.
By Property \ref{G}, we have $$(aa)_p=|R_0|+|\F[1,p-1|_a\times|R_1(aa)|+|\F[1,p-1|_b\times|R_2(aa)|+1.$$
Thus $\{(aa)_p\mid p\geq1\}=\{3,8,11,16,21,24,\ldots\}$.
Define a sequence $P_{aa}=\{P_{aa}[j]\}_{j\geq1}$ as below, called the
\emph{position sequence} of factor $aa$.
\begin{equation*}
P_{aa}[j]=\begin{cases}
1,&\text{if }j\in\{(aa)_p\mid p\geq1\},\\
0,&\text{otherwise}.
\end{cases}
\end{equation*}
Then $P_{aa}=001000010\ldots$, see the second line in Figure \ref{Fig:8}.
Moreover the number of repeated $aa$ in $\F[i;n]$ is $\Sigma P_{aa}[i;n-|aa|+1]$.

\medskip

\noindent\textbf{Remark 1.} For any factor $\w\prec\F$, $\w\prec\F[i;n]$ is equivalent to $\Sigma P_{\w}[i;n-|\w|+1]\geq1$. Now we consider that how many distinct squares in $\F[i;n]$ for $i,n\geq1$?
A direct method is that: we find out all squares in $\F$ denoted by set $\mathcal{S}$, and count $\sum_{\w\prec\mathcal{S}}\min\{1,\Sigma P_{\w}[i;n-|\w|+1]\}$.
But unfortunately, the set $\mathcal{S}$ is infinity.
The counting will be complicated and difficult to find
an explicit expression.

\noindent\textbf{Remark 2.} A known result is that all squares in $\F$ are of length $2f_m$ for $m\geq0$, see \cite{WW1994} for instance.
Thus in order to count the number of squares in $\F[i;n]$, we only need to
count the number of squares of length $2f_m$ in $\F[i;n]$ for each $m\geq0$ first.

\medskip

\textbf{Q2. How many squares of length $2f_m$ in $\F[i;n]$ for $i,n\geq1$?}

We consider factor spectrum
$\mathcal{S}_{\PP_2}=\{(\w,j)\mid \F[j;2f_m]=\w\w\}$, i.e.,
the set of all squares of length $2f_m$ with their positions in $\F$.
By Wen-Wen \cite{WW1994},
the number of distinct squares of length $2f_m$ in $\F$ is exact $f_m$ for $m\geq0$.
We will use a tool called \emph{kernel word} introduced and studied in \cite{HW2015-1, HW2015-2, HW2016-3} (we call them singular word in \cite{WW1994}).
By this way, we can divided all squares of length $2f_m$ into two types (for details see Subsection \ref{PS}.1), and then get the counting expressions in Section \ref{SecB}. %finally

The $m$-th kernel word $K_m$ is defined by
$K_m=F_{m+1}[f_{m+1}]F_m[1,f_{m}-1]$ for $m\geq-1$, by convention, $K_{-2}=\varepsilon$.
Let $\Ker(\w)$ be the maximal kernel word occurring in the factor $\w$,
called the \emph{kernel} of $\w$. By Theorem~1.9 in \cite{HW2015-1}, $\Ker(\w)$ occurs in $\w$ uniquely and only once.
In Subsection \ref{PS}.1, we prove that $\{\ \Ker(\w^2)\mid |\w|=f_m,~\w^2\prec\F\}=\{K_{m+1},K_{m-1}\}$ for $m\geq2$, that is, only two kernel words.
Thus all squares of length $2f_m$ can be divided into two types according to their kernels.
By Properties \ref{P} and \ref{wp} below, we can determine the positions of squares in each type by the positions of their kernels.

\begin{property}[Property 4.1 in \cite{HW2016-3}]\label{P}\
For $m\geq-1$, $p\geq1$ and $\phi=\frac{\sqrt{5}-1}{2}$,
the position of the $p$-th occurrence of $K_m$, denoted by $(K_m)_p$, is equal to
$(K_m)_p=pf_{m+1}+\lfloor p\phi \rfloor f_{m}.$
In particular,
$a_p=p+\lfloor p\phi \rfloor$,
$b_p=2p+\lfloor p\phi \rfloor$ for $p\geq1$, and $(K_m)_1=f_{m+1}$ for $m\geq-1$.
\end{property}

\begin{property}[Theorem 2.8 in \cite{HW2015-1}]\label{wp}\
Let $\w\prec \F$ be a factor. For all $p\geq 1$, the difference
$$(\Ker(\w))_p-(\w)_p=(\Ker(\w))_1-(\w)_1$$
is independent of $p$.
\end{property}

Take $m=2$ for instance.
There are three different squares of length $2f_2=6$ in $\F$: $abaaba=(aba)^2$, $baabaa=(baa)^2$ and $aabaab=(aab)^2$.
We define $\mathrm{code}(aba)=1$, $\mathrm{code}(baa)=2$ and $\mathrm{code}(aab)=3$.
Define the position sequence $\of=\{\of[j]\}_{j\geq1}$ as below.
\begin{equation*}
\of[j]=\begin{cases}
\mathrm{code}(\F[j;3]),&\text{if }\F[j;3]=\F[j+3;3],\\%\in\{aba,baa,aab\}
0,&\text{if }\F[j;3]\neq\F[j+3;3].
\end{cases}
\end{equation*}
Obviously, it shows the positions of all squares of length 6, see the 4-th line in Figure \ref{Fig:8}.
Thus the number of distinct squares of length 6 in $\F[i;n]$ is equal to the number of distinct letters in $\of[i;n-2f_2+1]$. It is equal to $\#\of[i;n-5]$.

%Define sequence $\oof=\{\oof[j]\}_{j\geq1}$ by $\oof[j]=\min\{1,\of[j]\}$, see lines 4 and 7 in Figure \ref{Fig:8}.
%The number of repeated squares of length 6 in $\F[i;n]$ is equal to the number of 1's occurring in $\oof[i;n-2f_2+1]$. It is equal to $\Sigma\oof[i;n-5]$.

Now we have two position sequences: $P_{\w}$ and $\of$.
$P_{\w}$ can be determined by the property of return word sequence.
%, and it's obviously to determine $\oof$ from $\of$ by $\oof[j]=\min\{1,\of[j]\}$.
Using the method below, we can determine $\of$ by $P_{\w}$.
By the definition of kernel,
$\Ker(abaaba)=aa$ and $\Ker(baabaa)=\Ker(aabaab)=aabaa$.
According to this, we split $\of$ into ${}^1\of+{}^2\of$, see lines 4 to 6 in Figure \ref{Fig:8}.
Here ${}^1\of$ and ${}^2\of$ are the position sequences of $abaaba$ and $\{baabaa,aabaab\}$ respectively.
By Property \ref{wp}, the difference $(aa)_p-(abaaba)_p=2$.
So the relation between ${}^1\of$ and $P_{aa}$ is
${}^1\of[j]=P_{aa}[j+2]\times \mathrm{code}(aba)$, see lines 2 and 5 in Figure \ref{Fig:8}. Similarly,
${}^2\of[j]=P_{aabaa}[j+1]\times\mathrm{code}(baa)+P_{aabaa}[j]\times\mathrm{code}(aab)$, see lines 3 and 6 in Figure \ref{Fig:8}.

\begin{figure}[!ht]
\centering
\setlength{\unitlength}{0.83mm}
\begin{picture}(205,70)%75
\small
\linethickness{1pt}
\put(10,66){$\F[1,31]$}
\put(25,66){$=$}
\put(30,66){${\begin{array}{*{32}{p{0.3cm}}}
a&b&a&a&b&a&b&a&a&b&a&a&b&a&b&a&a&b&a&b&a& a&b&a&a&b&a&b&a&a&b\end{array}}$}
\put(37.5,69){\line(1,0){8}}
\put(58,69){\line(1,0){8}}
\put(70,69){\line(1,0){8}}
\put(91,69){\line(1,0){8}}
\put(111,69){\line(1,0){8}}
\put(123,69){\line(1,0){8}}
\put(58,65){\line(1,0){20}}
\put(111,65){\line(1,0){20}}
\put(0.8,55){$P_{aabaa}[1,26]$}
\put(25,55){$=$}
\put(30,55){${\begin{array}{*{27}{p{0.3cm}}}
0&0&0&0&0&0&0&1&0&0&0&0&0&0&0&0&0&0&0&0&1&0&0&0&0&0\end{array}}$}
\put(6,60){$P_{aa}[1,26]$}
\put(25,60){$=$}
\put(30,60){${\begin{array}{*{27}{p{0.3cm}}}
0&0&1&0&0&0&0&1&0&0&1&0&0&0&0&1&0&0&0&0&1&0&0&1&0&0\end{array}}$}
\put(142,61.3){\vector(-1,0){4}}
\put(144,60){The position sequence of $aa$}
\put(142,56.3){\vector(-1,0){4}}
\put(144,55){The position sequence of $aabaa$}
\put(9.2,45){$\of[1,26]$}
\put(25,45){=}
\put(30,45){${\begin{array}{*{27}{p{0.3cm}}}
1&0&0&0&0&1&2&3&1&0&0&0&0&1&0&0&0&0&1&2&3&1&0&0&0&0\end{array}}$}
\put(10,40.5){$||$}
\put(7.4,35){${}^1\of[1,26]$}
\put(25,35){=}
\put(30,35){${\begin{array}{*{27}{p{0.3cm}}}
1&0&0&0&0&1&0&0&1&0&0&0&0&1&0&0&0&0&1&0&0&1&0&0&0&0\end{array}}$}
\put(151.4,35){$\Ker(abaaba)=aa$}
\put(148,36.3){\vector(-1,0){10}}
\put(9.5,30.5){+}
\put(7.4,25){${}^2\of[1,26]$}
\put(25,25){=}
\put(30,25){${\begin{array}{*{27}{p{0.3cm}}}
0&0&0&0&0&0&2&3&0&0&0&0&0&0&0&0&0&0&0&2&3&0&0&0&0&0\end{array}}$}
\put(148,25){$\begin{cases}
\Ker(baabaa)=aabaa\\
\Ker(aabaab)=aabaa
\end{cases}$}
\put(146,26.3){\vector(-1,0){8}}
%\put(9.2,10){$\oof[1,26]$}
%\put(25,10){=}
%\put(30,10){${\begin{array}{*{27}{p{0.3cm}}}
%1&0&0&0&0&1&1&1&1&0&0&0&0&1&0&0&0&0&1&1&1&1&0&0&0&0\end{array}}$}
%
\linethickness{0.2pt}
\put(29.5,19){\line(0,1){45}}
\put(50,19){\line(0,1){45}}
\put(62.5,19){\line(0,1){45}}
\put(83,19){\line(0,1){45}}
\put(103.5,19){\line(0,1){45}}
\put(115.5,19){\line(0,1){45}}
\put(135.5,19){\line(0,1){45}}
\end{picture}
\vspace{-2.4cm}
\caption{The relations among $\F$ and the position sequences $P_{aa}$, $P_{aabaa}$, $\of$.
\label{Fig:8}}%the Fibonacci sequence
\end{figure}

\textbf{Q3. How many $r$-powers of size $f_m$ in $\F[i;n]$ for $i,n\geq1$ and $r\geq2$?}

Taking $m=2$ for example, the key step is to determine the position sequence of $r$-powers of size $f_m=3$, denoted by $\of^r$, where $r=2+\frac{h}{3}$, $h=1,2,\ldots$.
Then the number of distinct $r$-powers of size $f_2$ in $\F[i;n]$ is $\#\of^r[i;n-rf_2+1]$.
%And the number of repeated $r$-powers of size $f_2$ in $\F[i;n]$ is $\Sigma\oof^r[i;n-rf_2+1]$ where $\oof^r[j]=\min\{1,\of^r[j]\}$.

Using the method below, we can determine $\of^r$ by $\of$ (the position sequence of squares of size 3).
Notice that, if $\of[j]\times\of[j+1]\neq0$, both $\F[j;6]$ and $\F[j+1;6]$ are squares. So $\F[j;7]$ is a $(2+\frac{1}{3})$-power.
This means $\of^{2+\frac{1}{3}}[j]\neq0$.
Similarly, if $\of[j]\times\of[j+1]=0$, then $\of^{2+\frac{1}{3}}[j]=0$, see the first two lines in Figure \ref{Fig:9}.
By an analogous argument, we can determine sequence $\of^{2+\frac{h+1}{3}}$ by $\of^{2+\frac{h}{3}}$ for $h=1,2,\ldots$, see lines 2 to 4 in Figure \ref{Fig:9}.
\begin{figure}[!ht]
\centering
\setlength{\unitlength}{0.83mm}
\begin{picture}(180,21)%75
\put(9.2,18){$\of[1,26]$}
\put(25,18){=}
\put(30,18){${\begin{array}{*{27}{p{0.3cm}}}
1&0&0&0&0&1&2&3&1&0&0&0&0&1&0&0&0&0&1&2&3&1&0&0&0&0\end{array}}$}
\put(7.3,12){$\of^{r}[1,26]$}
\put(25,12){=}
\put(30,12){${\begin{array}{*{27}{p{0.3cm}}}
0&0&0&0&0&1&2&3&0&0&0&0&0&0&0&0&0&0&1&2&3&0&0&0&0&0\end{array}}$ for $r=2+\frac{1}{3}$}
\put(7.3,6){$\of^{r}[1,26]$}
\put(25,6){=}
\put(30,6){${\begin{array}{*{27}{p{0.3cm}}}
0&0&0&0&0&1&2&0&0&0&0&0&0&0&0&0&0&0&1&2&0&0&0&0&0&0\end{array}}$ for $r=2+\frac{2}{3}$}
\put(7.3,0){$\of^{r}[1,26]$}
\put(25,0){=}
\put(30,0){${\begin{array}{*{27}{p{0.3cm}}}
0&0&0&0&0&1&0&0&0&0&0&0&0&0&0&0&0&0&1&0&0&0&0&0&0&0\end{array}}$ for $r=3$}
\end{picture}
\caption{The relations among sequences $\of$ and $\of^{r}$ for $r=2+\frac{h}{3}$, $h=1,2,3$.\label{Fig:9}}
\end{figure}

\textbf{Q4. The number of positions starting a square.}

Harju-K$\ddot{a}$rki-Nowotka \cite{HKN2011} considered the number of positions that do not start a square in a binary word. As an application of factor spectrum, we prove that all positions in the Fibonacci word start a square of length $2f_m$ for some $0\leq m\leq3$, see Property \ref{P5.2}.
Moreover, all positions in the Fibonacci word start infinite distinct squares, see Property \ref{P5.2.1}.

\section{Main results}\label{Sec.results}

In this paper, we count the number of distinct $r$-powers in
$\mathbb{F}[i;n]$ for all $i,n\geq1$ and $r\geq2$, denoted by $\mathrm{D}(r,i,n)$.
%and $\mathrm{R}(r,i,n)$ respectively.
We give precise results for $r\in\{2,2+\epsilon,3\}$ in this paper, where $\epsilon$ is a small positive number.
In this section, we show the results and some examples for $r=2$.
The results for $r\in\{2+\epsilon,3\}$ are given in Section \ref{Sec-r}.

\begin{theorem}[$\mathrm{D}(2,1,n)$]\label{P2.1}\
For $n\geq3$, let $h\geq0$ such that $2f_h\leq n<2f_{h+1}$.
Then the number of distinct squares in $\F[1,n]$ is equal to
$\mathrm{D}(2,1,n)=\min\{n-f_{h-1}-2,f_{h+1}+f_{h-1}-3\}$.
\end{theorem}

\noindent\textbf{Remark.}
Let $n=f_{h+2}$ for $h\geq0$. Since $2f_h\leq n<2f_{h+1}$ and $n-f_{h-1}-2\leq f_{h+1}+f_{h-1}-3$, $\mathrm{D}(2,1,n)=n-f_{h-1}-2=2f_{h}-2$.
We obtain thus the result of Fraenkel-Simpson \cite{FS1999}.

\begin{theorem}[$\mathrm{D}(2,i,n)$]\label{T2.1}\
For $i\geq1$ and $n\geq2$, let $h\geq0$ such that  $2f_h\leq n<2f_{h+1}$.
Then the number of distinct squares in $\F[i;n]$ is equal to
\begin{equation}\label{E3.5}
\mathrm{D}(2,i,n)=\begin{cases}
%0,&n=1,\\
[0,n-2,1,0,0][\hat{i}],&n=2,3,\\
[1,1,n-3,n-3,n-3,1,1,1][\hat{i}],&n=4,5.\\
[R_a,R_b][\hat{i}]+f_{h-1}-2,&n\geq6\text{ and }n\in \mathrm{I}\cup \mathrm{II}\cup \mathrm{III}\\
[R_a,R_b][\hat{j}]+f_{h-1}-2,&otherwise.
\end{cases}
\end{equation}
Here $\Fib(i-1)=a_ka_{k-1}\ldots a_1a_0$, $\hat{i}=[a_{h+1}a_{h}\ldots a_1a_0]_F+1$
and $\hat{j}=[a_{h+2}a_{h+1}\ldots a_1a_0]_F+1$.
The definitions of Ranges I to V see Figure \ref{Fig:1}(A). The expressions of $R_a$ and $R_b$ see Figure~\ref{Fig:1}(B).
\end{theorem}

\begin{figure}[!ht]
\centering
\scriptsize
\setlength{\unitlength}{0.82mm}
\begin{picture}(200,10)
\put(0,10){\textbf{(A) Divide $2f_h\leq n<2f_{h+1}$ into five ranges.}}
\put(0,5){Range I: $2f_h\leq n<2f_h+f_{h-3}-1$;}
\put(80,5){Range II: $2f_h+f_{h-3}-1\leq n<2f_h+f_{h-2}-1$;}
\put(0,0){Range III: $2f_h+f_{h-2}-1\leq n<f_{h+1}+2f_{h-1}$;}
\put(80,0){Range IV: $f_{h+1}+2f_{h-1}\leq n<3f_h$;}
\put(150,0){Range V: $3f_h\leq n<2f_{h+1}$.}
\end{picture}
\begin{picture}(200,105)
\put(0,100){\textbf{(B) The expressions of $R_a$ and $R_b$ in Theorem \ref{T2.1}.}}
\put(0,95){$n\in \mathrm{I}$:}
\put(13,95){$R_a=[\underbrace{n-2f_{h-1}\searrow}_{f_{h+1}+f_{h-1}-n-1},
\underbrace{2n-f_{h+1}-3f_{h-1}+1\longrightarrow}_{n-3f_{h-1}+1},
\underbrace{2n-f_{h+1}-3f_{h-1}+2\nearrow}_{f_{h+1}+f_{h-1}-n-2},
\underbrace{n-2f_{h-1}\longrightarrow}_{n-2f_{h}+2},
\underbrace{n-2f_{h-1}+1\nearrow}_{f_{h+1}+f_{h-1}-n-1}$,}
\put(22,85){$\underbrace{f_{h}\longrightarrow}_{n-2f_{h}+1},
\underbrace{f_{h}-1\searrow}_{f_{h-4}},
\underbrace{f_{h}-f_{h-4}-1\longrightarrow}_{f_{h+1}+f_{h-1}-n-1},
\underbrace{f_{h}-f_{h-4}\nearrow}_{f_{h-4}},
\underbrace{f_{h}\longrightarrow}_{n-2f_{h}+1},
\underbrace{f_{h}-1\searrow}_{f_{h+1}+f_{h-1}-n-1},
\underbrace{n-2f_{h-1}\longrightarrow}_{n-2f_{h}+1}]$;}
\put(13,75){$R_b=[\underbrace{n-2f_{h-1}\searrow}_{f_{h+1}+f_{h-1}-n-1},
\underbrace{2n-f_{h+1}-3f_{h-1}+1\longrightarrow}_{n-3f_{h-1}+1},
\underbrace{2n-f_{h+1}-3f_{h-1}+2\nearrow}_{f_{h+1}+f_{h-1}-n-2},
\underbrace{n-2f_{h-1}\longrightarrow}_{n-f_{h+1}+2}]$.}
\put(0,65){$n\in \mathrm{II}$:}
\put(13,65){$R_a=[\underbrace{n-2f_{h-1}\longrightarrow}_{f_{h+2}+f_{h-3}-n},
\underbrace{n-2f_{h-1}\searrow}_{n-f_{h+1}-f_{h-1}+1},
\underbrace{f_{h}\longrightarrow}_{f_{h-3}},
\underbrace{f_{h}-1\searrow}_{f_{h+1}+2f_{h-2}-n-2},
\underbrace{n-f_{h}-f_{h-2}+1\longrightarrow}_{n-f_{h+1}-f_{h-1}+3},$}
\put(22,55){$\underbrace{n-f_{h}-f_{h-2}+2\nearrow}_{f_{h+1}+2f_{h-2}-n-2},
\underbrace{f_{h}\longrightarrow}_{f_{h-3}},
\underbrace{f_{h}\nearrow}_{n-f_{h+1}-f_{h-1}+1},
\underbrace{n-2f_{h-1}\longrightarrow}_{f_{h-3}}];$}
\put(160,55){$R_b=[\underbrace{n-2f_{h-1}\longrightarrow}_{f_{h}}]$.}
\put(0,45){$n\in \mathrm{III}$:}
\put(13,45){$R_a=[\underbrace{n-2f_{h-1}\longrightarrow}_{f_{h+2}+f_{h-3}-n-1},
\underbrace{n-2f_{h-1}\searrow}_{n-f_{h+1}-f_{h-1}+1},
\underbrace{f_{h}\longrightarrow}_{3f_{h}-n-1},
\underbrace{f_{h}\nearrow}_{n-f_{h+1}-f_{h-1}+1},
\underbrace{n-2f_{h-1}\longrightarrow}_{f_{h-3}}]$;}
\put(160,45){$R_b=[\underbrace{n-2f_{h-1}\longrightarrow}_{f_{h}}]$;}
\put(0,35){$n\in \mathrm{IV}$:}
\put(13,35){$R_a=[\underbrace{f_{h+1}-1\searrow}_{f_{h-1}},
\underbrace{f_{h}\longrightarrow}_{f_{h+2}+f_{h-2}-n},
\underbrace{f_{h}+1\nearrow}_{n-f_{h+2}+f_{h-2}},
\underbrace{n-2f_{h-1}\longrightarrow}_{f_{h+1}+4f_{h-1}-n},
\underbrace{n-2f_{h-1}-1\searrow}_{n-2f_{h+1}+f_{h-2}}]$;}
\put(13,25){$R_b=[\underbrace{f_{h+1}-1\searrow}_{f_{h-1}},
\underbrace{f_{h}\longrightarrow}_{f_{h+2}+f_{h-2}-n},
\underbrace{f_{h}+1\nearrow}_{f_{h-1}-2},
\underbrace{f_{h+1}-1\longrightarrow}_{n-f_{h+2}+2}]$.}
\put(0,15){$n\in \mathrm{V}$:}
\put(13,15){$R_a=[\underbrace{f_{h+1}-1\searrow}_{f_{h+2}+f_{h}-n},
\underbrace{n-2f_{h}\longrightarrow}_{n-3f_{h}},
\underbrace{n-2f_{h}+1\nearrow}_{2f_{h-2}},
\underbrace{n-2f_{h-1}\longrightarrow}_{f_{h+2}+f_{h-1}-n},
\underbrace{n-2f_{h-1}-1\searrow}_{n-3f_{h}},
\underbrace{f_{h}+2f_{h-2}-1\longrightarrow}_{f_{h+2}+f_{h}-n-1},
\underbrace{f_{h}+2f_{h-2}\nearrow}_{n-3f_{h}+1},$}
\put(22,5){$\underbrace{n-2f_{h-1}\longrightarrow}_{f_{h+2}+f_{h-1}-n},
\underbrace{n-2f_{h-1}-1\searrow}_{n-2f_{h+1}+f_{h-2}}]$,}
\put(80,5){$R_b=[\underbrace{f_{h+1}-1\searrow}_{f_{h+2}+f_{h}-n},
\underbrace{n-2f_{h}\longrightarrow}_{n-3f_{h}},
\underbrace{n-2f_{h}+1\nearrow}_{f_{h+2}+f_{h}-n-2},
\underbrace{f_{h+1}-1\longrightarrow}_{n-f_{h+2}+2}]$.}
\end{picture}
\caption{The expressions of $R_a$ and $R_b$ in Theorem \ref{T2.1}.\label{Fig:1}}
\end{figure}

\noindent\textbf{Example.} We consider the number of distinct squares in $\F[333;20]$.
Since $n=20$, we have $h=4$. By the definitions of Ranges I to V in Figure \ref{Fig:1}(A), we have $n=20\in$ Range III=$\{18,19,\ldots,22\}$.
By the expressions of $R_a$ and $R_b$ in Figure~\ref{Fig:1}(B), for $n=20$ and $h=4$
\begin{equation*}
\begin{split}
R_a&=[\underbrace{n-2f_{h-1}\longrightarrow}_{f_{h+2}+f_{h-3}-n-1},
\underbrace{n-2f_{h-1}\searrow}_{n-f_{h+1}-f_{h-1}+1},
\underbrace{f_{h}\longrightarrow}_{3f_{h}-n-1},
\underbrace{f_{h}\nearrow}_{n-f_{h+1}-f_{h-1}+1},
\underbrace{n-2f_{h-1}\longrightarrow}_{f_{h-3}}]\\
&=[\underbrace{10\longrightarrow}_{2},
\underbrace{10\searrow}_{3},
\underbrace{8\longrightarrow}_{3},
\underbrace{8\nearrow}_{3},
\underbrace{10\longrightarrow}_{2}]
=[10,10,10,9,8,8,8,8,8,9,10,10,10];\\
R_b&=[\underbrace{n-2f_{h-1}\longrightarrow}_{f_{h}}]
=[\underbrace{10\longrightarrow}_{8}]=[10,10,10,10,10,10,10,10].
\end{split}
\end{equation*}
Since $i=333$ and $h=4$, $\Fib[i-1]=\Fib[332]=101000010010$, $\hat{i}=[a_{h+1}a_{h}\ldots a_1a_0]_F+1=[10010]_F+1=11$.
By Equation (\ref{E3.5}), $\mathrm{D}(2,333,20)=[R_a,R_b][\hat{i}]+f_{h-1}-2=10+5-2=13$.

\vspace{0.2cm}

\noindent\textbf{Remark.}
A famous conjecture of Fraenkel-Simpson \cite{FS1998}
states that a word of length $n$ contains fewer than $n$ distinct squares.
By Figure~\ref{Fig:1}(B), For fixed $n\geq2$ and $2f_h\leq n<2f_{h+1}$, the maximal element in $R_a$ and $R_b$ is $f_{h}$ for $n\in$ Range I, and $n-2f_{h-1}$ otherwise.
So

\begin{equation*}
\max\limits_{1\leq i\leq\infty}\mathrm{D}(2,i,n)=
\begin{cases}
f_{h}+f_{h-1}-2=f_{h+1}-2,&n\in\text{ Range I}\\
n-2f_{h-1}+f_{h-1}-2=n-f_{h-1}-2,&\text{otherwise}
\end{cases}
\end{equation*}
Thus $\mathrm{D}(2,i,n)<n$ for all $i,n\geq1$, and the conjecture holds for the Fibonacci word.

\begin{property}[]\
Any word of length $n$ in $\F$ contains fewer than $n$ distinct square factors.
\end{property}

In Section 6, we give some properties of function $D(2,i,n)$, such as:
(1) for $n$ large enough,
\begin{equation*}
\begin{cases}
0.7236n\leq \max\limits_{1\leq i\leq\infty}\mathrm{D}(2,i,n)\leq0.8090n,\quad
0.5393n\leq \min\limits_{1\leq i\leq\infty}\mathrm{D}(2,i,n)\leq0.6806n,\\
0.0528n\leq \max\limits_{1\leq i,j\leq\infty}\big|\mathrm{D}(2,i,n)-\mathrm{D}(2,j,n)\big|
\leq0.2547n;
\end{cases}
\end{equation*}
and (2) for fixed $n$ and any $M$ such that $\max\limits_{1\leq i\leq\infty}\mathrm{D}(2,i,n)\leq M\leq\min\limits_{1\leq i\leq\infty}\mathrm{D}(2,i,n)$,
there exists $i_0$ such that $D(2,i_0,n)=M$.

\section{The position sequences of squares}\label{PS}

In Section 2 (Q2), we claim that we can divided all squares of length $2f_m$ into two types by their kernels. We give the details in subsection below.

\subsection{Two cases of squares}

By Theorem~1.9 in \cite{HW2015-1} and Property \ref{G},
$$\{R_p(K_m),p\geq1\}=\{R_1(K_m),R_2(K_m)\}=\{K_mK_{m+1},K_mK_{m-1}\}.$$
Moreover, the return word sequence $\{R_p(K_m)\}_{p\geq1}$ is still a Fibonacci word
over $\{R_1(K_m),R_2(K_m)\}$ and $bb\not\prec\F$, so $R_2(K_m)R_2(K_m)\not\prec\F$.
Thus
\begin{equation*}
\begin{split}
&\{|R_p(K_m)R_{p+1}(K_m)|,p\geq1\}\\
=&\{|R_1(K_m)R_1(K_m)|,|R_1(K_m)R_2(K_m)|,|R_2(K_m)R_1(K_m)|\}
=\{2f_{m+2},f_{m+3}\}.
\end{split}
\end{equation*}

On the other hand, by Definition 2.9 and Corollary 2.10 in \cite{HW2015-1}, any factor $\w$ with kernel $K_m$ can be expressed uniquely as
$\w=K_{m+1}[i,f_{m+1}] K_m K_{m+1}[1,j],$ %=K_{m+3}[i,f_{m+2}+j]
where $2\leq i\leq f_{m+1}+1$ and $0\leq j\leq f_{m+1}-1$.
So $\w\w\prec\F$ means
$$K_{m+1}[i,f_{m+1}] \underbrace{K_m K_{m+1}[1,j]K_{m+1}[i,f_{m+1}]}_{\text{denoted by }R(K_m)} K_m K_{m+1}[1,j]\prec\F.$$
By the ranges of $i$ and $j$, $|\w|=|R(K_m)|=f_{m+2}+j-i+1\leq f_{m+3}-2$.

So $|R(K_m)|\leq \min\{2f_{m+2},f_{m+3}\}$. This means $K_m$ occurs in $R(K_m)$ only once.
Thus
$R(K_m)\in\{R_p(K_m),p\geq1\}=\{K_mK_{m+1},K_mK_{m-1}\}.$
Two cases have to be considered.

\medskip

\textbf{Case 1.} $R(K_m)=R_1(K_m)=K_mK_{m+1}$.
In this case $|R(K_m)|=|R_1(K_m)|$ implies $j=i-1$.
So $0\leq j\leq f_{m+1}-1$ gives a range of $i$. Comparing this range with $2\leq i\leq f_{m+1}+1$,
we have $2\leq i\leq f_{m+1}$ and $m\geq0$. Moreover $|\w|=|R(K_m)|=f_{m+2}$ and
\begin{equation*}
\begin{split}
&\w^2=
K_{m+1}[i,f_{m+1}] K_{m} K_{m+1} K_{m} K_{m+1}[1,i-1]\\
=&K_{m+2}[i,f_{m+2}]\underline{K_{m+1}} K_{m+2}[1,f_{m}+i-1]
=K_{m+4}[i,2f_{m+2}+i-1].
\end{split}
\end{equation*}
The second and third equalities hold by Lemma 2.2 in \cite{HW2015-1}, i.e.,
$$K_{m+2}=F_{m+1}[f_{m+1}]K_{m+1}[2,f_{m+1}]K_{m}
=K_{m}K_{m+1}[1,f_{m+1}-1]F_{m+1}[f_{m+1}]$$
and $K_{m+4}=K_{m+2}K_{m+1}K_{m+2}$.
Thus $K_{m+1}\prec\w\w\prec K_{m+4}[2,f_{m+4}-1]$.
By the cylinder structure of palindromes \cite{HW2016-3}, both $K_{m+2}$ and $K_{m+3}$ are not the factors of $\w\w$. Thus $K_{m+1}$ is the maximal kernel word in $\w\w$, i.e. $\Ker(\w\w)=K_{m+1}$ for $m\geq0$.% by the definition of kernel.

\medskip

\textbf{Case 2.} $R(K_m)=R_2(K_m)=K_mK_{m-1}$.
$|R(K_m)|=|R_2(K_m)|$ implies $j=i-f_m-1$.
Consequently, $f_m+1\leq i\leq f_{m+1}+1$ and $m\geq-1$.
Moreover $|\w|=|R(K_m)|=f_{m+1}$ and
\begin{equation*}
\begin{split}
&\w^2=
K_{m+1}[i,f_{m+1}] K_{m} K_{m-1} K_{m} K_{m+1}[1,i-f_m-1]\\
=&K_{m+3}[f_{m+2}+i,f_{m+3}] \underline{K_{m+2}} K_{m+3}[1,i-f_m-1]
=K_{m+5}[f_{m+2}+i,f_{m+3}+f_{m+1}+i-1].
\end{split}
\end{equation*}
Thus $K_{m+2}\prec\w\w\prec K_{m+5}[2,f_{m+5}-1]$.
By an analogous argument $\Ker(\w\w)=K_{m+2}$.

\medskip

Thus we get the expressions of all squares in the Fibonacci word.

\begin{definition}[]\ For all $m,p\geq1$,
%we define sets below. %two infinite sequences of sets below.
%Here we arrange all integers in the sets on the right hand in increasing order.
\begin{equation}\label{E2.1}
\begin{cases}
K^1_{m,p}=%\langle K^1_m,r,p\rangle
\{(\w^2)_p\mid \Ker(\w^2)=K_m,|\w|=f_{m+1},\w^2\prec\F\},\\
K^2_{m,p}=
\{(\w^2)_p\mid \Ker(\w^2)=K_m,|\w|=f_{m-1},\w^2\prec\F\}.
\end{cases}
\end{equation}
Each set contains several consecutive positive integers.
\end{definition}

%By the analysis in the first several paragraphs (three facts) in Section \ref{Sec-r}, $K^{j}_{r,m,p}$ contains the first several elements in $K^{j}_{2,m,p}$ for $j=1,2$.
%Thus $K^{j}_{r,m,p}$ ($j=1,2$) correspond to all positions of the two cases of $r$-powers.
%We consider $r=2$ first.
%\medskip

By Case 1 above,
$K^1_{m,p}=\{(\w^2)_p\mid \w^2=K_{m+1}[i,f_{m+1}]K_{m}K_{m+1}[1,f_{m-1}+i-1],
2\leq i\leq f_{m}\}$
for $m\geq1$. For each fixed $i$, $\Ker(\w^2)=K_m$. By Property \ref{wp},
$$(K_m)_p-(\w^2)_p=(K_m)_1-(\w^2)_1=|K_{m+1}[i,f_{m+1}]|=f_{m+1}-i+1.$$
So $K^1_{m,p}=\{(K_m)_p-f_{m+1}+i-1\mid 2\leq i\leq f_{m}\}$ for $m\geq1$.
Similarly, for $m\geq1$
\begin{equation*}
\begin{split}
K^2_{m,p}
&=\{(\w^2)_p\mid \w^2=K_{m+1}[f_{m}+i,f_{m+1}] K_{m} K_{m+1}[1,i-f_{m-2}-1],
f_{m-2}+1\leq i\leq f_{m-1}+1\}\\
&=\{(K_m)_p-f_{m-1}+i-1\mid f_{m-2}+1\leq i\leq f_{m-1}+1\}.
\end{split}
\end{equation*}

Thus for $m,p\geq1$
\begin{equation}\label{E2.2}
\begin{cases}
K^1_{m,p}=\{(K_m)_p-f_{m+1}+1,\ldots,(K_m)_p-f_{m-1}-1\},
&\#K^1_{m,p}=f_{m}-1,\\
K^2_{m,p}=\{(K_m)_p-f_{m-3},\ldots,(K_m)_p\},
&\#K^2_{m,p}=f_{m-3}+1.
\end{cases}
\end{equation}
%Each set contains several consecutive positive integers.
By Property \ref{P}, $(K_m)_1=f_{m+1}$. Thus for $m\geq1$
$$K^1_{m,1}=\{1,\ldots,f_{m}-1\}\text{ and }
K^2_{m,1}=\{f_{m}+f_{m-2},\ldots,f_{m+1}\}.$$

For instance, $(K_1)_1=3$,  $(K_1)_2=8$,  $(K_2)_1=5$ and $(K_2)_2=13$. Thus
the first few values of $K^1_{m,p}$ and $K^2_{m,p}$ are:
$K^1_{1,1}=\{1\}$, $K^1_{1,2}=\{6\}$, $K^1_{2,1}=\{1,2\}$, $K^1_{2,2}=\{9,10\}$,
$K^2_{1,1}=\{3\}$, $K^2_{1,2}=\{8\}$, $K^2_{2,1}=\{4,5\}$, $K^2_{2,2}=\{12,13\}$.

%\begin{center}
%\begin{tabular}{|l|l|l|l|}
%\hline
%$K^1_{1,1}=\{1\}$&$K^1_{1,2}=\{6\}$&$K^1_{2,1}=\{1,2\}$&$K^1_{2,2}=\{9,10\}$\\\hline
%$K^2_{1,1}=\{3\}$&$K^2_{1,2}=\{8\}$&$K^2_{2,1}=\{4,5\}$&$K^2_{2,2}=\{12,13\}$\\\hline
%\end{tabular}
%\end{center}

\subsection{The position sequences of squares}

Now we consider all squares of length $2f_m$.
For $m\geq2$, there %are $\#K^1_{2,m-1,p}+\#K^2_{2,m+1,p}=f_{m}$ distinct squares of length $2f_m$. There
are two types of distinct squares of length $2f_m$, whose positions are in sets $K^1_{m-1,p}$ and $K^2_{m+1,p}$, respectively.
%We consider the positions of each type of squares first.

\medskip

\textbf{1.}
Recall that $K^1_{m-1,1}=\{1,2,\ldots,f_{m-1}-1\}$ for $m\geq2$.
We denote square $\F[j;2f_m]$ by code $j$ for $1\leq j\leq f_{m-1}-1$.
We define sequence $^1\of^{2,m}$ for $i\geq1$ and $1\leq j\leq f_{m-1}-1$ that
\begin{equation*}
^1\of^{2,m}[i]=
\begin{cases}
j,&\text{if }\F[i;2f_m]=\F[j;2f_m],\\
0,&\text{otherwise.}
\end{cases}
\end{equation*}
We call it the position sequences of squares of length $2f_m$ with kernel $K_{m-1}$ for $m\geq2$.
%Since $K^1_{2,m-1,1}=\{1,\ldots,f_{m-1}-1\}$,

By Equation (\ref{E2.2}) and Property \ref{P}, $^1\of^{2,m}$ is a Fibonacci word over the alphabet $\{^1\!R^{2,m}_a,{}^1\!R^{2,m}_b\}$ with prefix $^1\!R^{2,m}_0$ for $m\geq2$, where
\begin{equation}\label{E2.3}
\begin{cases}
\begin{aligned}
^1\!R^{2,m}_a&={}^1\of^{2,m}[\min K^1_{m-1,1},\min K^1_{m-1,2}-1]
={}^1\of^{2,m}[1,f_{m+1}]\\
&=[1,2,\ldots,f_{m-1}-1,\underbrace{0,\ldots,0}_{f_{m}+1}],\\
^1\!R^{2,m}_b&={}^1\of^{2,m}[\min K^1_{m-1,2},\min K^1_{m-1,3}-1]
={}^1\of^{2,m}[f_{m+1}+1,f_{m+2}]\\
&={}^1\of^{2,m}[f_{m+1}+1;f_{m}]
=[1,2,\ldots,f_{m-1}-1,\underbrace{0,\ldots,0}_{f_{m-2}+1}],\\
^1\!R^{2,m}_0&={}^1\of^{2,m}[1,\min K^1_{m-1,1}-1]=\varepsilon.
\end{aligned}
\end{cases}
\end{equation}

\medskip

\textbf{2.}
Recall that $K^2_{m+1,1}=\{f_{m+1}+f_{m-1},\ldots,f_{m+2}\}$ for $m\geq0$.
We denote square $\F[f_{m+1}+j;2f_m]$ by code $j$ for $f_{m-1}\leq j\leq f_{m}$.
We define sequence $^2\of^{2,m}$ for $i\geq1$ and $f_{m-1}\leq j\leq f_{m}$ that
\begin{equation*}
{}^2\of^{2,m}[i]=
\begin{cases}
j,&\text{if }\F[i;2f_m]=\F[f_{m+1}+j;2f_m],\\
0,&\text{otherwise.}
\end{cases}
\end{equation*}
We call it the position sequences of squares of length $2f_m$ with kernel $K_{m+1}$ for $m\geq0$.

By Equation (\ref{E2.2}) and Property \ref{P}, $^2\of^{2,m}$ is a Fibonacci word over the alphabet $\{\tilde{R}^{2,m}_a,\tilde{R}^{2,m}_b\}$ with prefix $\tilde{R}^{2,m}_0$ for $m\geq2$, where
\begin{equation*}
\begin{cases}
\begin{aligned}
\tilde{R}^{2,m}_a&={}^2\of^{2,m}[\min K^2_{m+1,1},\min K^2_{m+1,2}-1]
={}^2\of^{2,m}[f_{m+1}+f_{m-1},4f_{m+1}-1]\\
&={}^2\of^{2,m}[f_{m+1}+f_{m-1};f_{m+3}]
=[f_{m-1},f_{m-1}+1,\ldots,f_{m},\underbrace{0,\ldots,0}_{2f_{m+1}+f_{m-1}-1}],\\
\tilde{R}^{2,m}_b&={}^2\of^{2,m}[\min K^2_{m+1,2},\min K^2_{m+1,3}-1]
={}^2\of^{2,m}[4f_{m+1},f_{m+2}+4f_{m+1}-1]\\
&={}^2\of^{2,m}[4f_{m+1};f_{m+2}]
=[f_{m-1},f_{m-1}+1,\ldots,f_{m},\underbrace{0,\ldots,0}_{f_{m+1}+f_{m-1}-1}],\\
\tilde{R}^{2,m}_0&={}^2\of^{2,m}[1,\min K^2_{m+1,1}-1]
={}^2\of^{2,m}[1,f_{m+1}+f_{m-1}-1]=[\underbrace{0,\ldots,0}_{f_{m+1}+f_{m-1}-1}].
\end{aligned}
\end{cases}
\end{equation*}

Notice that both $\tilde{R}^{2,m}_a$ and $\tilde{R}^{2,m}_b$ have
suffix $\tilde{R}^{2,m}_0$, we can
shift to left each return word $f_{m+1}+f_{m-1}-1$ letters. Thus we get the
equivalent expressions that $\hat{R}^{2,m}_0=\varepsilon$ and
\begin{equation*}
\begin{cases}
\hat{R}^{2,m}_a=[\underbrace{0,\ldots,0}_{f_{m+1}+f_{m-1}-1},f_{m-1},f_{m-1}+1,\ldots,f_{m},\underbrace{0,\ldots,0}_{f_{m+1}}],\\
\hat{R}^{2,m}_b=[\underbrace{0,\ldots,0}_{f_{m+1}+f_{m-1}-1},f_{m-1},f_{m-1}+1,\ldots,f_{m}].
\end{cases}
\end{equation*}
Furthermore, the Fibonacci morphism $\sigma$ over the alphabet $\{{}^2\!R^{2,m}_a,{}^2\!R^{2,m}_b\}$ where
\begin{equation}\label{E2.4}
{}^2\!R^{2,m}_a=[\underbrace{0,\ldots,0}_{f_{m+1}}]\text{ and }
{}^2\!R^{2,m}_b=[\underbrace{0,\ldots,0}_{f_{m-1}-1},f_{m-1},f_{m-1}+1,\ldots,f_{m}]
\end{equation}
satisfies that $\sigma^2({}^2\!R^{2,m}_a)=\hat{R}^{2,m}_a$ and $\sigma^2({}^2\!R^{2,m}_b)=\hat{R}^{2,m}_b$.
This means $^2\of^{2,m}=\F({}^2\!R^{2,m}_a,{}^2\!R^{2,m}_b)$.
%is a Fibonacci word over the alphabet $\{{}^2\!R^{2,m}_a,{}^2\!R^{2,m}_b\}$.

\bigskip

By the argument above, we give the definition of position sequences of squares.

\begin{definition}[]\label{D3.4}
We define sequence $\of^{2,m}$ ($m\geq0$) that for $i\geq1$ and $1\leq j\leq f_{m}$
\begin{equation*}
\of^{2,m}[i]=
\begin{cases}
j,&\text{if }\F[i;2f_m]=\F[j;2f_m]\text{ for }1\leq j\leq f_{m-1}-1,\\
j,&\text{if }\F[i;2f_m]=\F[f_{m+1}+j;2f_m]\text{ for }f_{m-1}\leq j\leq f_{m},\\
0,&\text{otherwise.}
\end{cases}
\end{equation*}
We call $\of^{2,m}$ the position sequences of squares of length $2f_m$ for $m\geq0$.
\end{definition}

When $m\geq2$,
if there exist an integer $p$ such that $i\in K^1_{m-1,p}$ (resp. $i\in K^2_{m+1,p}$), $\F[i;2f_m]$ is a square with kernel $K_{m-1}$ (resp. $K_{m+1}$).
Since the kernel of a factor is unique \cite{HW2015-1}, sets $K^1_{m-1,p}$ and $K^2_{m+1,q}$ are disjoint for all $p,q\geq1$.
This means ${}^1\of^{2,m}[i]\times{}^2\of^{2,m}[i]=0$ for $i\geq1$.
Thus $\of^{2,m}={}^1\of^{2,m}+{}^2\of^{2,m}$.
By Equations (\ref{E2.3}) and (\ref{E2.4}),
$\of^{2,m}=\F(R^{2,m}_a,R^{2,m}_b)$ where
% is a Fibonacci word over the alphabet $\{R^{2,m}_a,R^{2,m}_b\}$
\begin{equation}\label{E2.5}
\begin{cases}
R^{2,m}_a={}^1\!R^{2,m}_a+{}^2\!R^{2,m}_a=[1,2,\ldots,f_{m-1}-1,\underbrace{0,\ldots,0}_{f_{m}+1}],\\
R^{2,m}_b={}^1\!R^{2,m}_b+{}^2\!R^{2,m}_b=[1,2,\ldots,f_{m}].
\end{cases}
\end{equation}

When $m=0,1$, there are $\#K^2_{m+1,p}=f_{m-2}+1$ distinct squares of length $2f_m$.
Thus sequence $\of^{2,m}$ is exact sequence ${}^2\of^{2,m}$.
But since all digits in ${}^1\!R^{2,m}_a$ and ${}^1\!R^{2,m}_b$ are zero for $m=0,1$, the expressions in Equation (\ref{E2.5}) still hold.

\begin{theorem}[]\label{P2.5}\
$\of^{2,m}$ is a Fibonacci word over %the alphabet
$\{R^{2,m}_a,R^{2,m}_b\}$ given in Equation (\ref{E2.5}) for $m\geq0$.
\end{theorem}

\section{The number of squares in $\F[i;n]$}\label{SecB}

%First we consider the number of squares of length $2f_m$ in $\F[i;n]$.
By the definition of sequence $\of^{2,m}$, $\of^{2,m}[i]=j$ ($1\leq j\leq f_m$) means $\F[i;2f_m]$ is a square of length $2f_m$ encoded by $j$; and $\of^{2,m}[i]=0$ means $\F[i;2f_m]$ is not a square.
Thus in order to count the number of distinct squares of length $2f_m$ in $\F[i;n]$, we only need to consider the number of distinct letters in $\of^{2,m}[i;n-2f_m+1]$, i.e., $\#\of^{2,m}[i;n-2f_m+1]$.

When $n<2f_{h+1}$, there is no square of length $2f_m$ ($m\geq h+1$) in $\F[i;n]$.
Thus we only need to consider sequences $\of^{2,m}$ for $0\leq m\leq h$ when $2f_h\leq n<2f_{h+1}$.

%First we give a lemma about $\F$ using Zeckendorf numeration system.

\subsection{Three useful lemmas}

Let $R_a$ and $R_b$ be two words of lengths $f_m$ and $f_{m-1}$ for $m\geq0$.
We give three useful lemmas about the sequence $\F(R_a,R_b)$ in this subsection.

\begin{lemma}\label{L1}\ For $i> f_m$, let $\Fib(i-1)=a_ka_{k-1}\ldots a_{1}a_0$ then $k\geq m$ and
$$\F(R_a,R_b)[i]=(R_aR_b)\big[[a_{m}a_{m-1}\ldots a_0]_F+1\big].$$
\end{lemma}

\noindent\textbf{Example.} Take $i=20$, $R_a=ABC$ and $R_b=DE$. Then $\Fib(i-1)=101001$.
Moreover, $\F(R_a,R_b)=ABC DE ABC ABC DE ABC DE ABC\cdots$, so $\F(R_a,R_b)[i]=B$.
On the other hand, $m=2$ and
$$(R_aR_b)\big[[a_{m}a_{m-1}\ldots a_0]_F+1\big]=(ABCDE)\big[[001]_F+1\big]=ABCDE[2]=B.$$

\begin{lemma}\label{L2}\
Let $h\geq0$ be an integer, and
\begin{equation*}
M=\begin{cases}
h,&h\geq m+2,\\
m+2,&h=m+1,\\
m+1,&h=1,2,\ldots,m,\\
m,&h=0.
\end{cases}
\end{equation*}
Then the sequence $\Theta=\{\F(R_a,R_b)[i;f_h]\}_{i\geq1}$ is a Fibonacci word over $\{\Theta[1,f_{M}],\Theta[f_{M}+1,f_{M+1}]\}$ without prefix. % the alphabet
\end{lemma}

\noindent\textbf{Example.} Let $m=h=1$. By Lemma \ref{L2}, $M=m+1=2$ and
the sequence $\{\F(ab,c)[i;2]\}_{i\geq1}$ is a Fibonacci word over $\{\Theta[1,f_{M}],\Theta[f_{M}+1,f_{M+1}]\}=\{ABC,AD\}$
where $A=ab$, $B=bc$, $C=ca$ and $D=ba$. We check this conclusion in Figure \ref{Fig:11}.

\begin{figure}[!h]
\centering
\setlength{\unitlength}{0.83mm}
\begin{picture}(152,60)%75
\put(19,55){$\F(ab,c)$}
\put(35,55){$=$}
\put(40,55){${\begin{array}{*{15}{p{0.6cm}}}
a&b&c&a&b&a&b&c&a&b&c&a&b&a&...\end{array}}$}
\put(0,48){$\{\F(ab,c)[i;2]\}_{i\geq1}$}
\put(35,48){$=$}
\put(40,48){${\begin{array}{*{15}{p{0.6cm}}}
ab&bc&ca&ab&ba&ab&bc&ca&ab&bc&ca&ab&ba& &...\end{array}}$}
\put(40,41){${\begin{array}{*{15}{p{0.6cm}}}
A&B&C&A&D&A&B&C&A&B&C&A&D&&...\end{array}}$}
\put(40,41){$\underbrace{\hspace{1.6cm}}_{ABC}$}
\put(63,41){$\underbrace{\hspace{1.0cm}}_{AD}$}
\put(79,41){$\underbrace{\hspace{1.6cm}}_{ABC}$}
\put(102,41){$\underbrace{\hspace{1.6cm}}_{ABC}$}
\put(125,41){$\underbrace{\hspace{1.0cm}}_{AD}$}
\put(7,41){$\F(ABC,AD)$}
\put(35,41){$=$}
\put(148.2,41){$...$}
\end{picture}
\vspace{-3cm}
\caption{For instance, the sequence $\{\F(ab,c)[i;2]\}_{i\geq1}=\F(ABC,AD)$.
\label{Fig:11}}
\end{figure}

\begin{proof}
(1) We first let $h=m+2$.
Since $\F=\F(F_2,F_1)$, a decomposition of $\F$ is
\begin{equation}\label{E5.2}
\F=F_2\cdot F_1\cdot F_2\cdot F_2\cdot F_1\cdots
\end{equation}
Here  $\w\cdot\nu$ is the concatenation of two words $\w$ and $\nu$.
In this decomposition, $F_2$ and $F_1$ are used as building blocks of $\F$.
Since $bb\not\!\prec\F$, each occurrence of letter $a$ in $\F$ is followed by word $ba$ or $a$. Thus each block $F_2$ in Equation (\ref{E5.2}) is followed by the blocks $F_1\cdot F_2=ababa$ or $F_2=aba$. This means, each block $F_2=aba$ is followed by word $aba$.
Notice that, the last word $aba$ may not be a block.
Similarly, $\F(R_a,R_b)$ has a decomposition with blocks $R_aR_bR_a$ and $R_aR_b$ that
\begin{equation}\label{E5.3}
\F(R_a,R_b)=R_aR_bR_a\cdot R_aR_b\cdot R_aR_bR_a\cdot R_aR_bR_a\cdot R_aR_b\cdots
\end{equation}
And each block $R_aR_bR_a$ in Equation (\ref{E5.3}) is followed by word $R_aR_bR_a$ (may not be a block).
On the other hand, since $|R_aR_bR_a|=f_{m+2}$, each word starting in block $R_aR_bR_a$ of length $f_{m+2}$ belongs in $R_aR_bR_a\cdot R_aR_bR_a$.
This means, for fixed $1\leq i\leq f_{m+2}$, the factor of length $f_{m+2}$ starting in the $i$-th letter in any block $R_aR_bR_a$ are $(R_aR_bR_aR_aR_bR_a)[i;f_{m+2}]$.

Similarly, each block $R_aR_b$ is followed by the word $R_aR_bR_a$ too. And for $1\leq i\leq f_{m+1}$, the factor of length $f_{m+2}$ starting in the $i$-th letter in any block $R_aR_b$ are $(R_aR_bR_aR_bR_a)[i;f_{m+2}]$.
Thus the conclusion holds for $h=m+2$.

(2) Since $\F(R_a,R_b)$ is also a Fibonacci word over $\{R_aR_b,R_a\}$ and $|R_aR_b|=f_{m+1}$, the conclusion holds for $h=m+3$ by (1) above.
And the conclusion holds for all $h\geq m+3$ by induction.

(3) When $h=m+1$, $\F(R_a,R_b)[i;f_h]$ is prefix of $\F(R_a,R_b)[i;f_{m+2}]$. By (1), the conclusion holds.

(4) When $h=m$, $\F(R_a,R_b)$ has a decomposition
$\F(R_a,R_b)=R_aR_b\cdot R_a\cdot R_aR_b\cdot R_aR_b\cdot R_a\cdots$.
Each block $R_aR_b$ (resp. $R_a$) is followed by the word $R_a$.
On the other hand, $|R_a|=f_m$. This means, for $1\leq i\leq f_{m+1}$ (resp. $1\leq i\leq f_{m}$), the factor of length $f_{m}$ starting in the $i$-th letter in any block $R_aR_b$ (resp. $R_a$) are $(R_aR_bR_a)[i;f_{m}]$ (resp. $(R_aR_a)[i;f_{m}]$).
Thus the conclusion holds for $h=m$.

(5) When $h=1,\ldots,m-1$, by (4) and an analogous argument in (2), the conclusion holds too.
\end{proof}

\begin{lemma}\label{L3}\
Let $2f_h\leq n<2f_{h+1}$ and $0\leq m\leq h$.
%We consider the sequence $\Theta=\of^{2,m}[i;n-2f_m+1]$.
$\Theta=\{\of^{2,m}[i;n-2f_m+1]\}_{i\geq1}$ is a Fibonacci word over $\{\Theta[1,f_{h+3}],\Theta[f_{h+3}+1,f_{h+4}]\}$ without prefix.

%(2) For $i> f_{h+3}$, let $\Fib(i-1)=a_ka_{k-1}\ldots a_{1}a_0$ then $$\of^{2,m}[i;n-2f_m+1]=\of^{2,m}\big[[a_{h+3}a_{h+2}\ldots a_1a_0]_F+1;n-2f_m+1\big].$$
\end{lemma}

\begin{proof}
By Theorem \ref{P2.5}, $\of^{2,m}$ is a Fibonacci word over $\{\of^{2,m}[1,f_{m+1}],\of^{2,m}[f_{m+1}+1,f_{m+2}]\}$.
On the other hand, $n-2f_m+1<n<2f_{h+1}<f_{h+3}$. Thus by Lemma \ref{L2} the conclusion holds. %Moreover, the conclusion (2) holds by Lemma \ref{L1}.
\end{proof}

\subsection{Distinct squares in $\F[i;n]$}\label{Dist}

Since the return word sequence of letter $b$ in $\F$
is $\F(baa,ba)$,
two successional occurrences of letter $b$ have two cases: $baab$ and $bab$.
On the other hand, $\of^{2,m}$ is a Fibonacci word over
$$\{R^{2,m}_a,R^{2,m}_b\}
=\{[1,2,\ldots,f_{m-1}-1,\underbrace{0,\ldots,0}_{f_{m}+1}],[1,2,\ldots,f_{m}]\}.$$
Two successional occurrences of $R^{2,m}_b$ have two cases:
$[R^{2,m}_b,R^{2,m}_a,R^{2,m}_a,R^{2,m}_b]$ and $[R^{2,m}_b,R^{2,m}_a,R^{2,m}_b]$.
This means when $n=|[R^{2,m}_b,R^{2,m}_a,R^{2,m}_a]|=f_{m+3}$,
one of the two cases below holds

\begin{equation*}
\begin{cases}
R^{2,m}_b[1,j]\triangleright\of^{2,m}[i;n]\text{ and }
R^{2,m}_b[j+1,f_m]\triangleleft\of^{2,m}[i;n]\\
\hspace{6cm}\text{ for some integer }j\text{ such that }0\leq j\leq f_m,\\
R^{2,m}_b\prec\of^{2,m}[i;n].
\end{cases}
\end{equation*}
Thus when $n-2f_m+1\geq f_{m+3}$, $\of^{2,m}[i;n-2f_m+1]$ contains all elements in $R^{2,m}_b$.
By Equation (\ref{E2.5}), $R^{2,m}_b$ gives positions of all distinct squares of length $2f_m$.
In this situation $\#\of^{2,m}[i;n-2f_m+1]=f_m$, and
$\F[i;n]$ contains all distinct squares of length $2f_m$.

This means $\#\of^{2,m}[i;n-2f_m+1]=f_m$ for $0\leq m\leq h-3$. So
\begin{equation*}
\mathrm{D}(2,i,n)=\sum\nolimits_{m=0}^{h-3}f_m+\sum\nolimits_{m=h-2}^{h}\#\of^{2,m}[i;n-2f_m+1].
\end{equation*}

%By analogous arguments, when $1\leq n\leq3$, $h=0$; when $n=4,5$, $h=1$.
Since $\sum_{j=0}^{m} f_{j}=f_{m+2}-2$,
the number of distinct squares in $\F[i;n]$ is equal to
%Equation (\ref{E3.1}) can be simplified as
\begin{equation}\label{E3.4}
\mathrm{D}(2,i,n)=\begin{cases}
%0,&n=1,\\
\#\of^{2,0}[i;n-1],&n=2,3,\\
\#\of^{2,0}[i;n-1]+\#\of^{2,1}[i;n-3],&n=4,5,\\
\sum_{m=h-2}^{h}\#\of^{2,m}[i;n-2f_m+1]+f_{h-1}-2,&n\geq6.
\end{cases}
\end{equation}

Now we turn to simplify $\sum_{m=h-2}^{h}\#\of^{2,m}[i;n-2f_m+1]$ for $2f_h\leq n<2f_{h+1}$.
First we divide $2f_h\leq n<2f_{h+1}$ into five ranges, see Figure \ref{Fig:1}(A).

\medskip

\textbf{Conclusion 1.} By Lemma \ref{L3}, $\{\#\of^{2,m}[i;n-2f_m+1]\}_{i\geq1}$
is a Fibonacci word over $\{\tilde{R}_a,\tilde{R}_b\}$ for fixed $m\in\{h,h-1,h-2\}$ and fixed $n$, where $|\tilde{R}_a|=f_{h+3}$ and $|\tilde{R}_b|=f_{h+2}$.
We must determine the expressions of $\tilde{R}_a$ and $\tilde{R}_b$ for each $m\in\{h,h-1,h-2\}$.

Since the proof is tediously long, we only list the expressions of $R_a$ and $R_b$ in Figure~\ref{Fig:10}.
\begin{equation*}
\begin{cases}
\tilde{R}_a=R_aR_bR_a,\tilde{R}_b=R_aR_b,|R_a|=f_{h+1},|R_b|=f_{h},
&\text{if }n\in \mathrm{I}\cup \mathrm{II}\cup \mathrm{III}(1),\\
\tilde{R}_a=R_aR_b,\tilde{R}_b=R_a,|R_a|=f_{h+2},|R_b|=f_{h+1},&\text{otherwise}.
\end{cases}
\end{equation*}
Thus in all cases above, $\{\#\of^{2,m}[i;n-2f_m+1]\}_{i\geq1}$ is a Fibonacci word over $\{R_a,R_b\}$ without prefix.

%When $n\in \mathrm{I}\cup \mathrm{II}\cup \mathrm{III}(1)$, $|R_a|=f_{h+1}$ and $|R_b|=f_{h}$;
%otherwise $|R_a|=f_{h+2}$ and $|R_b|=f_{h+1}$.
%The proof is tediously long, so we only give some key conclusions.

\textbf{Conclusion 2.} By Conclusion 1, $\{\sum_{m=h-2}^{h}\#\of^{2,m}[i;n-2f_m+1]\}_{i\geq1}$
is also a Fibonacci word for fixed $n$.
We give the expressions of $R_a$ and $R_b$ in Figure~\ref{Fig:1}(B).
When $n\in \mathrm{I}\cup \mathrm{II}\cup \mathrm{III}$, $|R_a|=f_{h+1}$ and $|R_b|=f_{h}$;
otherwise $|R_a|=f_{h+2}$ and $|R_b|=f_{h+1}$.

\begin{figure}[!t]
\centering
\scriptsize
\setlength{\unitlength}{0.82mm}
\begin{picture}(205,255)
%\put(0,250){\textbf{(B) The expressions of $R_a$ and $R_b$ in Conclusion 1 for $r=2$.}}
\put(0,245){First we divide Range III into Range III(1) ($2f_h+f_{h-2}-1\leq n<f_{h+2}-1$) and
Range III(2) ($f_{h+2}-1\leq n<f_{h+1}+2f_{h-1}$).}
\put(0,240){\textbf{When} $\mathbf{m=h}$.}
\put(25,240){$n\in \mathrm{I}\cup \mathrm{II}\cup \mathrm{III}(1)$:}
\put(55,240){$R_a=[\underbrace{n-2f_h+1\longrightarrow}_{f_{h+2}-n-1},
\underbrace{n-2f_h\searrow}_{n-2f_h},
\underbrace{0\longrightarrow}_{3f_h-n+1},
\underbrace{1\nearrow}_{n-2f_h}]$;}
\put(168,240){$R_b=[\underbrace{n-2f_h+1\longrightarrow}_{f_{h}}]$.}
\put(0,230){$n\in\mathrm{III}(2)$:}
\put(18,230){$R_a=[\underbrace{f_{h-1}-1\searrow}_{f_{h-1}-1},
\underbrace{0\longrightarrow}_{3f_h-n+1},\underbrace{1\nearrow}_{n-2f_h},
\underbrace{n-2f_h+1\longrightarrow}_{f_{h+2}+f_h-n-1},
\underbrace{n-2f_h\searrow}_{n-f_{h+2}+1}]$;}
\put(124,230){$R_b=[\underbrace{f_{h-1}-1\searrow}_{f_{h-1}-1},
\underbrace{0\longrightarrow}_{3f_h-n+1},\underbrace{1\nearrow}_{f_{h-1}-1},
\underbrace{f_{h-1}-1\longrightarrow}_{n-f_{h+2}+1}]$.}
\put(0,220){$n\in\mathrm{IV}$:}
\put(13,220){$R_a=[\underbrace{f_{h-1}-1\searrow}_{f_{h-1}-1},
\underbrace{0\longrightarrow}_{3f_h-n+1},\underbrace{1\nearrow}_{n-2f_h},
\underbrace{n-2f_h+1\longrightarrow}_{f_{h+2}+f_h-n-1},
\underbrace{n-2f_h\searrow}_{n-f_{h+2}+1}]$;}
\put(124,220){$R_b=[\underbrace{f_{h-1}-1\searrow}_{f_{h-1}-1},
\underbrace{0\longrightarrow}_{3f_h-n+1},
\underbrace{1\nearrow}_{f_{h-1}-2},
\underbrace{f_{h-1}-1\longrightarrow}_{n-f_{h+2}+2}]$.}
\put(0,210){$n\in\mathrm{V}$:}
\put(13,210){$R_a=[\underbrace{f_{h-1}-1\searrow}_{f_{h+2}+f_h-n-1},
\underbrace{n-3f_h\longrightarrow}_{n-3f_h+1},
\underbrace{n-3f_h+1\nearrow}_{4f_h-n-1},
\underbrace{f_h\longrightarrow}_{n-3f_h+f_{h-1}+1},
\underbrace{f_h-1\searrow}_{f_{h-2}}]$;}
\put(13,200){$R_b=[\underbrace{f_{h-1}-1\searrow}_{f_{h+2}+f_h-n-1},
\underbrace{n-3f_h\longrightarrow}_{n-3f_h+1},
\underbrace{n-3f_h+1\nearrow}_{f_{h+2}+f_h-n-2},
\underbrace{f_{h-1}-1\longrightarrow}_{n-f_{h+2}+2}]$.}
\put(0,190){\textbf{When} $\mathbf{m=h-1}$.}
\put(32,190){$n\in \mathrm{I}$:}
\put(45,190){$R_a=[\underbrace{f_{h-2}-1\searrow}_{f_{h+1}+f_{h-1}-n-1},
\underbrace{n-3f_{h-1}\longrightarrow}_{n-3f_{h-1}+1},
\underbrace{n-3f_{h-1}+1\nearrow}_{4f_{h-1}-n-1},
\underbrace{f_{h-1}\longrightarrow}_{n-2f_{h-1}-f_{h-3}+1},
\underbrace{f_{h-1}-1\searrow}_{f_{h-3}}]$;}
\put(45,180){$R_b=[\underbrace{f_{h-2}-1\searrow}_{f_{h+1}+f_{h-1}-n-1},
\underbrace{n-3f_{h-1}\longrightarrow}_{n-3f_{h-1}+1},
\underbrace{n-3f_{h-1}+1\nearrow}_{f_{h+1}+f_{h-1}-n-2},
\underbrace{f_{h-2}-1\longrightarrow}_{n-f_{h+1}+2}]$.}
\put(0,170){$n\in \mathrm{II}\cup\mathrm{III}(1)$:}
\put(25,170){$R_a=[\underbrace{f_{h-2}-1\longrightarrow}_{f_{h+2}-n-1},
\underbrace{f_{h-2}\nearrow}_{f_{h-3}},
\underbrace{f_{h-1}\longrightarrow}_{n-2f_{h-1}-f_{h-3}+1},
\underbrace{f_{h-1}-1\searrow}_{f_{h-3}}]$;}
\put(170,170){$R_b=[\underbrace{f_{h-2}-1\longrightarrow}_{f_{h}}]$.}
\put(0,160){$n\in \mathrm{III}(2)$:}
\put(18,160){$R_a=[\underbrace{n-f_{h+1}-f_{h-1}+1\nearrow}_{f_{h+2}+f_{h-3}-n-1},
\underbrace{f_{h-1}\longrightarrow}_{n-2f_{h-1}-f_{h-3}+1},
\underbrace{f_{h-1}-1\searrow}_{f_{h-3}},
\underbrace{f_{h-2}-1\longrightarrow}_{f_{h+2}+f_{h}-n-1},
\underbrace{f_{h-2}\nearrow}_{n-f_{h+2}+1}]$;}
\put(13,150){$R_b=[\underbrace{n-f_{h+1}-f_{h-1}+1\nearrow}_{f_{h+2}+f_{h-3}-n-1},
\underbrace{f_{h-1}\longrightarrow}_{n-2f_{h-1}-f_{h-3}+1},
\underbrace{f_{h-1}-1\searrow}_{f_{h+2}+f_{h-3}-n-2},
\underbrace{n-f_{h+1}-f_{h-1}+1\longrightarrow}_{n-f_{h+2}+2}]$.}
\put(0,140){$n\in \mathrm{IV}\cup\mathrm{V}$:}
\put(20,140){$R_a=[\underbrace{f_{h-1}\longrightarrow}_{f_{h+1}-f_{h-3}},
\underbrace{f_{h-1}-1\searrow}_{f_{h-3}},
\underbrace{f_{h-2}-1\longrightarrow}_{f_{h+2}+f_{h}-n-1},
\underbrace{f_{h-2}\nearrow}_{f_{h-3}},
\underbrace{f_{h-1}\longrightarrow}_{n-f_{h+2}-f_{h-3}+1}]$;}
\put(170,140){$R_b=[\underbrace{f_{h-1}\longrightarrow}_{f_{h+1}}]$.}
\put(0,130){\textbf{When} $\mathbf{m=h-2}$.}
\put(32,130){$n\in \mathrm{I}$:}
\put(45,130){$R_a=[\underbrace{f_{h-2}\longrightarrow}_{f_{h-1}+f_{h-3}},
\underbrace{f_{h-2}-1\searrow}_{f_{h-4}},
\underbrace{f_{h-3}-1\longrightarrow}_{f_{h+1}+f_{h-1}-n-1},
\underbrace{f_{h-3}\nearrow}_{f_{h-4}},
\underbrace{f_{h-2}\longrightarrow}_{n-f_{h}-2f_{h-2}+1}]$;}
\put(170,130){$R_b=[\underbrace{f_{h-2}\longrightarrow}_{f_{h}}]$.}
\put(0,120){$n\in \mathrm{II}$:}
\put(13,120){$R_a=[\underbrace{f_{h-2}\longrightarrow}_{f_{h-1}+f_{h-3}},
\underbrace{f_{h-2}-1\searrow}_{2f_{h}+f_{h-2}-n-1},
\underbrace{n-2f_{h}+1\longrightarrow}_{n-f_{h+1}-f_{h-1}+1},
\underbrace{n-2f_{h}+1\nearrow}_{2f_{h}+f_{h-2}-n-1},
\underbrace{f_{h-2}\longrightarrow}_{n-f_{h}-2f_{h-2}+1}]$;}
\put(170,120){$R_b=[\underbrace{f_{h-2}\longrightarrow}_{f_{h}}]$.}
\put(0,110){$n\in \mathrm{III}(1)$:}
\put(18,110){$R_a=[\underbrace{f_{h-2}\longrightarrow}_{f_{h+1}}]$;}
\put(48,110){$R_b=[\underbrace{f_{h-2}\longrightarrow}_{f_{h}}]$.}
\put(105,110){$n\in\mathrm{III}(2)\cup\mathrm{IV}\cup\mathrm{V}$:}
\put(140,110){$R_a=[\underbrace{f_{h-2}\longrightarrow}_{f_{h+2}}]$;}
\put(170,110){$R_b=[\underbrace{f_{h-2}\longrightarrow}_{f_{h+1}}]$.}
\end{picture}
\vspace{-8.8cm}
\caption{The expressions of $R_a$ and $R_b$ in Conclusion 1 for $r=2$.\label{Fig:10}}
\end{figure}

\textbf{Conclusion 3.} By Conclusion 2 and Lemma \ref{L1},
\begin{equation*}
\sum_{m=h-2}^{h}\#\of^{2,m}[i;n-2f_m+1]=
\begin{cases}
[R_a,R_b][\hat{i}],&\text{if }n\in \mathrm{I}\cup \mathrm{II}\cup \mathrm{III},\\
[R_a,R_b][\hat{j}],&\text{otherwise}.
\end{cases}
\end{equation*}
Here $\Fib(i-1)=a_ka_{k-1}\ldots a_1a_0$, $\hat{i}=[a_{h+2}a_{h+1}\ldots a_1a_0]_F+1$
and $\hat{j}=[a_{h+3}a_{h+2}\ldots a_1a_0]_F+1$.

By an analogous argument, we simplify Equation (\ref{E3.4}) for $2\leq n\leq5$.
Thus we get Theorem \ref{T2.1}.

\subsection{Distinct squares in $\F[1,n]$}

By analyzing the first number for each $n$ in Figure \ref{Fig:1}(B), we have
\begin{equation*}
\mathrm{D}(2,1,n)-f_{h-1}+2=\begin{cases}
n-2f_{h-1},&2f_h\leq n<f_{h+1}+2f_{h-1}~(\text{Range I}\cup\text{II}\cup\text{III}),\\
f_{h+1}-1,&f_{h+1}+2f_{h-1}\leq n<2f_{h+1}~(\text{Range V}\cup\text{VI}).
\end{cases}
\end{equation*}

Let $Z[n]=\mathrm{D}(2,1,n+1)-\mathrm{D}(2,1,n)$, $Z_1=Z[1,4]$ and $Z_h=Z[2f_h-1,2f_{h+1}-2]$ for $h\geq2$.
Then $\mathrm{D}(2,1,n)=\sum_{j=1}^{n-1}Z[j]$.
We can prove that $Z_1=[0,0,1,0]$ for $h\geq2$
$$Z_h=[\underbrace{1,\ldots,1}_{f_{h-1}+f_{h-3}},\underbrace{0,\ldots,0}_{f_{h-2}}]\text{ and}
\sum_{j=1}^{2f_h-2}Z[j]=f_h+f_{h-2}-3.$$
Thus when $2f_h-1\leq n<2f_{h+1}-1$,
$\mathrm{D}(2,1,n)=\sum_{j=1}^{n-1}Z[j]=\sum_{j=1}^{2f_h-2}Z[j]+\sum_{j=2f_h-1}^{n-1}Z[j]$. So
$$\mathrm{D}(2,1,n)=f_h+f_{h-2}-3+\min\{n-2f_h+1,f_{h-1}+f_{h-3}\}.$$
Thus $\mathrm{D}(2,1,n)=\min\{n-f_{h-1}-2,f_{h+1}+f_{h-1}-3\}$.
Now we obtain Theorem \ref{P2.1}.

\subsection{All positions starting a square}

Harju-K$\ddot{a}$rki-Nowotka \cite{HKN2011} considered the number $N(\w)$ of positions that do not start a square in a binary word $\w$.
Letting $N(n)$ be the maximum of $N(\w)$ for length $|\w|=n$,
they showed that $\lim\frac{N(n)}{n}=\frac{15}{31}$.
As an application of factor spectrum, we have

\begin{property}[]\label{P5.2}\
All positions in $\F$ start a square of length $2f_m$ for some $0\leq m\leq3$.
\end{property}

Define sequence $\mathbb{D}$ by that $\mathbb{D}[i]=\sum_{m=0}^3\of^{2,m}[i]$.
The property is equivalent to prove $\mathbb{D}[i]>0$ for all $i$.
Since all of $\of^{2,m}$ ($0\leq m\leq3$) are Fibonacci words over $\{\of^{2,m}[1,f_4],\of^{2,m}[f_4+1,f_5]\}$, we only need to prove that $\mathbb{D}[i]>0$ for all $1\leq i\leq f_5=13$.

By the expressions of $R^{2,m}_a$ and $R^{2,m}_b$ given in Theorem \ref{P2.5} and Equation (\ref{E2.5}),
\begin{equation*}
\begin{cases}
\of^{2,0}[1,13]=0010000100100,&\of^{2,1}[1,13]=0001200000012,\\
\of^{2,2}[1,13]=1000012310000,&\of^{2,3}[1,13]=1200000012345.
\end{cases}
\end{equation*}
This obtain the conclusion in Property \ref{P5.2}.

\noindent\textbf{Remark.} Let us consider factor spectrum
$\mathcal{S}_{\PP_3}=\{(\w,j)\mid \F[j;2|\w|]=\w\w\}$.
This property means that for all $j\geq1$ there exists a factor $\w$ of length $f_m$ ($0\leq m\leq3$) such that $(\w,j)\in\mathcal{S}_{\PP_3}$.

%\medskip

Furthermore, for all $i\geq1$ there exists $m$ such that $i\leq f_{m-1}-1$. Since
$$R^{2,m}_a=[1,2,\ldots,f_{m-1}-1,\underbrace{0,\ldots,0}_{f_{m}+1}],$$
we have $\of^{2,h}[i]\neq0$ for all $h\geq m$. This means

\begin{property}[]\label{P5.2.1}\
All positions in $\F$ start infinite distinct squares.
\end{property}

\section{Analyses of $\mathrm{D}(2,i,n)$}%

We give the explicit expression of $\mathrm{D}(2,i,n)$
in Theorem \ref{T2.1} and Figure \ref{Fig:1}.
$\mathrm{D}(2,i,n)$ is a function of two variables: $i$ and $n$.
For fixed $n$, $\mathrm{D}(2,i,n)$ and $\mathrm{D}(2,j,n)$ maybe different for $i\neq j$. This means, the numbers of distinct squares in $\F[i;n]$ and $\F[j;n]$
maybe different.
For instance, let $n=16$, $i=2$ and $j=7$. Then $F[2;16]=baababaabaababaa$, $F[7;16]=baabaababaababaa$ and
\begin{equation*}
\begin{cases}
\mathrm{D}(2,2,16)=\#\{aa,abab,baba,abaaba,baabaa,aabaab,
baababaaba,baababaabaababaa\}=8;\\
\mathrm{D}(2,7,16)=\#\{aa,abab,baba,abaaba,baabaa,aabaab,\\
\hspace{2.72cm}abaababaab,baababaaba,aababaabab,ababaababa,babaababaa\}=11.
\end{cases}
\end{equation*}

In this section, we consider some properties of $\mathrm{D}(2,i,n)$ for fixed $n$ and different $i$, such as:
maximum and minimum values, asymptotic properties and property of intermediate values.

\subsection{Maximum and minimum values}

By Theorem \ref{T2.1} and Figure \ref{Fig:1}(B), $\max\limits_{1\leq i\leq\infty}\mathrm{D}(2,i,n)=\max\limits_{1\leq i\leq |R_a|+|R_b|}\mathrm{D}(2,i,n)$.
Thus there exists $i_1$ such that $\mathrm{D}(2,i_1,n)=\max\limits_{1\leq i\leq\infty}\mathrm{D}(2,i,n)$.
Similarly, there exists $i_2$ such that $\mathrm{D}(2,i_2,n)=\min\limits_{1\leq i\leq\infty}\mathrm{D}(2,i,n)$.
In Section \ref{Sec.results} we show that

\begin{equation}\label{max}
\max\limits_{1\leq i\leq\infty}\mathrm{D}(2,i,n)-f_{h-1}+2=
\begin{cases}
f_{h},&n\in\text{ Range I},\\
n-2f_{h-1},&\text{otherwise}.
\end{cases}
\end{equation}
We divide Range I into Range I(1) ($2f_h\leq n<\frac{7}{2}f_{h-1}$) and Range I(2) ($\frac{7}{2}f_{h-1}\leq n<2f_h+f_{h-3}-1$).
Similarly, by Figure \ref{Fig:1}(B) we have
\begin{equation}\label{min}
\min\limits_{1\leq i\leq\infty}\mathrm{D}(2,i,n)-f_{h-1}+2=
\begin{cases}
2n-f_{h+1}-3f_{h-1}+1,&n\in\text{ Range I(1)},\\
f_{h}-f_{h-4}-1,&n\in\text{ Range I(2)},\\
n-f_{h}-f_{h-2}+1,&n\in\text{ Range II},\\
f_{h},&n\in\text{ Range III}\cup \text{IV},\\
n-2f_{h},&n\in\text{ Range V}.\\
\end{cases}
\end{equation}

Thus we get an approximate solution of $D(2,i,n)$: $n-f_{h+1}$. In fact,
for $n\geq2$, let $h\geq0$ such that $2f_h\leq n<2f_{h+1}$,
then $\big|n-f_{h+1}-D(2,i,n)\big|\leq f_{h-1}$ for all $i\geq1$.

\subsection{Asymptotic properties}

Let $\alpha=\frac{\sqrt{5}+1}{2}\doteq1.6180$. A known result is $\lim_{h\rightarrow\infty}\frac{f_{h+1}}{f_h}=\alpha$. We consider six asymptotic values:
\begin{equation*}
\begin{cases}
\overline{Max}:=\overline{\lim}_{n\rightarrow\infty}\dfrac{\max\limits_{1\leq i\leq\infty}\mathrm{D}(2,i,n)}{n},~
&\underline{Max}:=\underline{\lim}_{n\rightarrow\infty}\dfrac{\max\limits_{1\leq i\leq\infty}\mathrm{D}(2,i,n)}{n},\\
\overline{Min}:=\overline{\lim}_{n\rightarrow\infty}\dfrac{\min\limits_{1\leq i\leq\infty}\mathrm{D}(2,i,n)}{n},~
&\underline{Min}:=\underline{\lim}_{n\rightarrow\infty}\dfrac{\min\limits_{1\leq i\leq\infty}\mathrm{D}(2,i,n)}{n},\\
\overline{\Delta}:=\overline{\lim}_{n\rightarrow\infty}
\dfrac{\max\limits_{1\leq i,j\leq\infty}\left|\mathrm{D}(2,i,n)-\mathrm{D}(2,j,n)\right|}{n},
&\underline{\Delta}:=\underline{\lim}_{n\rightarrow\infty}
\dfrac{\max\limits_{1\leq i,j\leq\infty}\left|\mathrm{D}(2,i,n)-\mathrm{D}(2,j,n)\right|}{n}.
\end{cases}
\end{equation*}

By Equation (\ref{max}),

When $n\in$ Range I, $\overline{Max}=\overline{\lim}_{n\rightarrow\infty}\frac{f_{h+1}-2}{n}
=\lim_{h\rightarrow\infty}\frac{f_{h+1}-2}{2f_h}=\frac{\alpha}{2}\doteq0.8090$

\hspace{3.4cm}$\underline{Max}=\underline{\lim}_{n\rightarrow\infty}\frac{f_{h+1}-2}{n}=
\lim_{h\rightarrow\infty}\frac{f_{h+1}-2}{2f_h+f_{h-3}-1}=\frac{\alpha^4}{2\alpha^3+1}\doteq0.7236$

Otherwise, $\frac{\max\limits_{1\leq i\leq\infty}\mathrm{D}(2,i,n)}{n}=\frac{n-f_{h-1}}{n}=1-\frac{f_{h-1}}{n}$. Thus

\hspace{1.9cm}$\overline{Max}=1-\lim_{h\rightarrow\infty}\frac{f_{h-1}}{2f_{h+1}-1}
=1-\frac{1}{2\alpha^2}\doteq0.8090$;

\hspace{1.9cm}$\underline{Max}=1-\lim_{h\rightarrow\infty}\frac{f_{h-1}}{2f_h+f_{h-3}-1}
=1-\frac{\alpha^2}{2\alpha^3+1}\doteq0.7236$.

This means %for $n$ large enough,
$0.7236n\leq \max\limits_{1\leq i\leq\infty}\mathrm{D}(2,i,n)\leq0.8090n$
for $n$ large enough.

By Equation (\ref{min}) and an analogous argument,
\begin{equation*}
\begin{cases}
\underline{Min}=2-\frac{\alpha^2+2}{2\alpha}\doteq0.5729,\hspace{0.5cm}
\overline{Min}=2-\frac{2}{7}(\alpha^2+2)\doteq0.6806,&n\in\text{ Range I(1)},\\
\underline{Min}=\frac{\alpha^5-1}{2\alpha^4+\alpha}\doteq0.6583,\hspace{0.98cm}
\overline{Min}=\frac{2(\alpha^5-1)}{7\alpha^3}\doteq0.6806,&n\in\text{ Range I(2)},\\
\underline{Min}=1-\frac{2\alpha}{2\alpha^3+1}\doteq 0.6583,\hspace{0.35cm}
\overline{Min}=1-\frac{2}{2\alpha^2+1}\doteq0.6793,&n\in\text{ Range II},\\
\underline{Min}=\frac{\alpha}{3}\doteq 0.5393,\hspace{1.7cm}
\overline{Min}=\frac{\alpha^3}{2\alpha^2+1}\doteq0.6793,&n\in\text{ Range III}\cup \text{IV},\\
\underline{Min}=1-\frac{\alpha^2+1}{3\alpha^2}\doteq 0.5393,\hspace{0.5cm}
\overline{Min}=1-\frac{\alpha^2+1}{2\alpha^3}\doteq0.5729,&n\in\text{ Range V}.
\end{cases}
\end{equation*}

Since $\min\{0.5729,0.6583,0.5393\}=0.5393$ and $\max\{0.6806,0.6793,0.5729\}=0.6806$, we have $0.5393n\leq \min\limits_{1\leq i\leq\infty}\mathrm{D}(2,i,n)\leq0.6806n$
for $n$ large enough.

%Finally we consider the upper bound of asymptotic solution $\Delta$. It is obviously less that $0.8090-0.5393=0.2697$. Moreover

Similarly,
\begin{equation*}
\begin{cases}
\underline{\Delta}=\frac{2\times(\alpha^3+3)}{7}-2\doteq0.0674,\hspace{0.5cm}
\overline{\Delta}=\frac{\alpha^3+3}{2\alpha}-2\doteq0.2361,
&n\in\text{ Range I(1)},\\
\underline{\Delta}=\frac{1}{2\alpha^4+\alpha}\doteq0.0653,\hspace{1.6cm}
\overline{\Delta}=\frac{2}{7\alpha^3}\doteq0.0675,
&n\in\text{ Range I(2)},\\
\underline{\Delta}=\frac{1}{2\alpha^4+2\alpha^2}\doteq0.0528,\hspace{1.3cm}
\overline{\Delta}=\frac{1}{2\alpha^4+\alpha}\doteq0.0653,
&n\in\text{ Range II},\\
\underline{\Delta}=1-\frac{\alpha^3+\alpha}{2\alpha^2+1}\doteq0.0613,\hspace{0.97cm}
\overline{\Delta}=1-\frac{\alpha^2+1}{3\alpha}\doteq0.2547,
&n\in\text{ Range III}\cup \text{IV},\\
\underline{\Delta}=\frac{1}{\alpha^3}\doteq0.2361,\hspace{2.16cm}
\overline{\Delta}=\frac{2}{3\alpha^2}\doteq0.2547,
&n\in\text{ Range V}.
\end{cases}
\end{equation*}
Since $\min\{0.0674,0.0653,0.0528,0.0613,0.2361\}=0.0528$ and
$\max\{0.2361,0.0675,0.0653,0.2547\}=0.2547$, we have
$0.0528n\leq  \max\limits_{1\leq i,j\leq\infty}\big|\mathrm{D}(2,i,n)-\mathrm{D}(2,j,n)\big|
\leq0.2547n$ for $n$ large enough.

\subsection{Property of intermediate values}

By Figure \ref{Fig:1}(B) and Theorem \ref{T2.1}, %when $n\in$ Range II to V,
$\big|D(2,i+1,n)-D(2,i,n)\big|\leq 1$. Thus for fixed $n$ and any $M$ such that $\max\limits_{1\leq i\leq\infty}\mathrm{D}(2,i,n)\leq M\leq\min\limits_{1\leq i\leq\infty}\mathrm{D}(2,i,n)$,
there exists $i_0$ such that $D(2,i_0,n)=M$.

\section{The number of r-powers in $\F[i;n]$}\label{Sec-r}

First of all, we prove some facts, which are fit for all sequence including $\F$.

Fact 1. If both $P$ and $P+1$ are the positions of some squares of size $n$, then $P$ is the position of some $(2+\frac{1}{n})$-powers of size $n$.
In fact, since $P$ is the position of some squares of size $n$, $\F[P,P+2n-1]$ has expression $x_1x_2\cdots x_nx_1x_2\cdots x_n$ for $x_i\in\{a,b\}$. Since $P+1$ is the position of some squares of size $n$ too, $\F[P+2n]=x_1$. This means $\F[P,P+2n]=x_1x_2\cdots x_nx_1x_2\cdots x_nx_1$.
Thus $P$ is the position of the $(2+\frac{1}{n})$-power factor $\F[P,P+2n]$.

Fact 2. If both $P$ and $P+1$ are the positions of some $(2+\frac{i}{n})$-powers of size $n$, then $P$ is the position of some $(2+\frac{i+1}{n})$-powers of size $n$ for $i\geq0$.

Fact 3. If each integer in set $\{P,P+1,\ldots,P+h\}$ is the position of some $(2+\frac{i}{n})$-powers of size $n$, then each integer in %interval
$\{P,P+1,\ldots,P+h-j\}$ is the position of some $(2+\frac{i+j}{n})$-powers of size $n$ for $h\geq j$. %for $i,j\geq0$.
Now we consider the sequences $\of^{r,m}$. Let $i\geq0$ and $n,h\geq1$ such that
\begin{equation}\label{E5.1}
\begin{split}
&\of^{2+\frac{i}{f_m},m}[n;h+2]=[0,1,2,\ldots,h,0]\\
\text{then }&
\of^{2+\frac{i+j}{f_m},m}[n;h+2]
=[0,1,2,\ldots,h-j,\underbrace{0,\ldots,0}_{j+1}]
\end{split}
\end{equation}
for $0<j\leq h$.
For real number $r\geq2$, $r$-power $\w^r$ of size $|\w|$ means $h$-power $\w^h$, where $h=\frac{\lceil r|\w|\rceil}{|\w|}$.
So we can generalize the conclusion above to all real number $r\geq2$.

\bigskip

We first generalize the definition of $\of^{2,m}$ (see Definition \ref{D3.4}) to
$\of^{r,m}$ as below.

\begin{definition}[]\label{D6.1}
We define sequence $\of^{r,m}$ ($r\geq2$ and $m\geq0$) that for $i\geq1$ and $1\leq j\leq f_{m}$
\begin{equation*}
\of^{r,m}[i]=
\begin{cases}
j,&\text{if }\of^{2,m}[i]=j\text{ and }\F[i;\lceil rf_m\rceil]\text{ is an }r\text{-power};\\
0,&\text{otherwise.}
\end{cases}
\end{equation*}
We call $\of^{r,m}$ the position sequence of $r$-powers of size $f_m$ for $m\geq0$.
\end{definition}

By Fact 3 above and Theorem \ref{P2.5}, we have

\begin{property}[]\label{P5.1}\
$\of^{r,m}$ is a Fibonacci word over %the alphabet
$\{R^{r,m}_a,R^{r,m}_b\}$ for $r\geq2$ and $m\geq0$. Specifically,

\emph{(1)} $\of^{r,0}$  is a zero sequence for $r>2$;

\emph{(2)} $\of^{r,1}$ is a zero sequence for $r>2.5$;
$R^{r,1}_a=[0,0,0]$ and $R^{r,1}_b=[1,0]$ for $2<r\leq2.5$;

\emph{(3)} $\of^{r,m}$ $(m\geq2)$ is a zero sequence for $r>3+\frac{f_{m-1}-1}{f_m}$; and
\begin{equation*}
\begin{cases}
R^{r,m}_a=[1,2,\ldots,f_{m-1}-x,\underbrace{0,\ldots,0}_{f_{m}+x}],~
R^{r,m}_b=[1,2,\ldots,f_{m}],
&2\leq r\leq2+\frac{f_{m-1}-1}{f_m}\\
R^{r,m}_a=[\underbrace{0,\ldots,0}_{f_{m+1}}],~
R^{r,m}_b=[1,2,\ldots,f_{m}-y,\underbrace{0,\ldots,0}_{y}],
&2<r-\frac{f_{m-1}-1}{f_m}\leq 3.
\end{cases}
\end{equation*}
Here $x=\lceil rf_m\rceil-2f_m+1$ and $y=\lceil rf_m\rceil-f_{m+2}+1$.
\end{property}

\noindent\textbf{Remark.} Given a sequence $\rho$, it is an interesting and challenging task to determine its critical exponent $e$, such that $\rho$ contains $r$-powers for all $r<e$, but has no $r$-powers for all $r>e$ \cite{AS2003}.
The critical exponent of $\F$ is $\lim_{m\rightarrow\infty} 3+\frac{f_{m-1}-1}{f_m}=3+\frac{1}{\alpha}\approx3.618$, where $\alpha=\frac{1+\sqrt{5}}{2}$. This is
an immediate corollary of Property \ref{P5.1}.

\bigskip

Now we turn to consider two special cases: $r=2+\epsilon$ and $r=3$.

When $r=2+\epsilon$,
$\of^{r,0}$ is a zero sequence;
$R^{r,1}_a=[0,0,0]$ and $R^{r,1}_b=[1,0]$;
for $m\geq2$
\begin{equation}\label{r2e}
R^{r,m}_a=[1,2,\ldots,f_{m-1}-2,\underbrace{0,\ldots,0}_{f_{m}+2}]\text{ and }
R^{r,m}_b=[1,2,\ldots,f_{m}].
\end{equation}
When $r=3$, both $\of^{r,0}$ and $\of^{r,1}$ are zero sequences;
for $m\geq2$
\begin{equation}\label{r3}
R^{r,m}_a=[\underbrace{0,\ldots,0}_{f_{m+1}}]\text{ and }
R^{r,m}_b=[1,2,\ldots,f_{m-1}-1,\underbrace{0,\ldots,0}_{f_{m-2}+1}].
\end{equation}

\subsection{Distinct r-powers for $r=2+\epsilon$}

For $n\geq9$, let $h\geq2$ such that $2f_h\leq n<2f_{h+1}$.
Then
$$\mathrm{D}(2+\epsilon,1,n)=\min\{n-f_{h-1}-6,f_{h+1}+f_{h-1}-6\}.$$

By Equation (\ref{r2e}), $\#\of^{2+\epsilon,1}[i;n-2f_1]=1$ and $\#\of^{2+\epsilon,m}[i;n-2f_m]=f_m$ for $2\leq m\leq h-3$.
Thus $\sum\nolimits_{m=1}^{h-3}\#\of^{2,m}[i;n-2f_m]=\sum\nolimits_{m=2}^{h-3}f_m+1=f_{h-1}-4$.
For $i\geq1$ and $n\geq5$, let $h\geq1$ such that $2f_h+1\leq n<2f_{h+1}+1$.
Then %the number of distinct $r$-powers in $\F[i;n]$ is equal to
\begin{equation*}
\mathrm{D}(2+\varepsilon,i,n)=\begin{cases}
%0,&n=1,\\
[0,0,n-5,1,0,0,0,0],&n=5,6,\\
[n-7,1,1,1,n-7,n-6,n-6,1,n-7,1,1,1,0][\hat{j}],&n=7,8,\\
[1,1,n-8,n-7,n-7,3,n-7,2,1,1,1,1,1][\hat{j}],&n=9,10,\\
[R_a,R_b][\hat{i}]+f_{h-1}-4,&\hspace{-3.8cm}n\geq11\text{ and }n\in \mathrm{I}\cup \mathrm{II}\cup \mathrm{III}\cup \mathrm{IV},\\
[R_a,R_b][\hat{j}]+f_{h-1}-4,&\hspace{-0.3cm}otherwise.
\end{cases}
\end{equation*}
Here $\Fib(i-1)=a_ka_{k-1}\ldots a_1a_0$, $\hat{i}=[a_{h+1}a_{h}\ldots a_1a_0]_F+1$
and $\hat{j}=[a_{h+2}a_{h+1}\ldots a_1a_0]_F+1$.
The definitions of Ranges I to VI, and the expressions of $R_a$ and $R_b$
see Figure~\ref{Fig:2}.

\begin{property}[]\
All positions in $\F$ starting a ($2+\epsilon$)-power of size $f_m$ for some $2\leq m\leq6$.
\end{property}

The proof could be obtained by an analogous argument in Property \ref{P5.2}.
Let us consider factor spectrum
$\mathcal{S}_{\PP_4}=\{(\w,j)\mid \F[j;2|\w|+1]=\w\w\cdot\w[1]\}$.
This property means that for all $j\geq1$ there exists a factor $\w$ such that $(\w,j)\in\mathcal{S}_{\PP_4}$.

\subsection{Distinct r-powers for $r=3$}

For $n\geq14$, let $h\geq3$ such that $2f_h-1\leq n<2f_{h+1}-1$.
Then
$$\mathrm{D}(3,1,n)=\max\{n-f_{h+1}-f_{h-1}-h+1,f_{h-1}-h\}.$$
By Equation (\ref{r3}), $\#\of^{2+\epsilon,m}[i;n-3f_m]=f_{m-1}-1$ for $2\leq m\leq h-3$.
Thus $$\sum\nolimits_{m=2}^{h-3}\#\of^{2,m}[i;n-3f_m]
=\sum\nolimits_{m=2}^{h-3}(f_{m-1}+1)=f_{h-1}+h-9.$$
For $i\geq1$ and $n\geq9$, let $h\geq2$ such that  $3f_h\leq n<3f_{h+1}$.
Then %the number of distinct cubes in $\F[i;n]$ is equal to
\begin{equation*}
\mathrm{D}(3,i,n)=\begin{cases}
[\underbrace{0,\ldots,0}_{14-n},\underbrace{1,\ldots,1}_{n-8},
\underbrace{0,\ldots,0}_{7},][\hat{i}],&9\leq n\leq 14,\\
[R_a,R_b][\hat{i}]+f_{h-1}+h-9,&n\geq15\text{ and }n\not\in\text{Range }\mathrm{X},\\
[R_a,R_b][\hat{j}]+f_{h-1}+h-9,&otherwise.
\end{cases}
\end{equation*}
Here $\Fib(i-1)=a_ka_{k-1}\ldots a_1a_0$, $\hat{i}=[a_{h+2}a_{h}\ldots a_1a_0]_F+1$
and $\hat{j}=[a_{h+3}a_{h+1}\ldots a_1a_0]_F+1$.
The definitions of Ranges I to X, and the expressions of $R_a$ and $R_b$
see Figure~\ref{Fig:3}.

\begin{figure}[!ht]
\centering
\scriptsize
\setlength{\unitlength}{0.82mm}
\begin{picture}(205,14)
%\put(0,15){\textbf{(A) The expressions of $R_a$ and $R_b$ in Conclusion 1 for $r=2+\epsilon$.}}
\put(0,10){Divide $2f_h+1\leq n<2f_{h+1}+1$ into six ranges:}
\put(0,5){Range I: $2f_h+1\leq n<2f_h+f_{h-3}$;}
\put(70,5){Range II: $n=2f_h+f_{h-3}$;}
\put(120,5){Range III: $2f_h+f_{h-3}+1\leq n<2f_h+f_{h-2}$;}
\put(0,0){Range IV: $2f_h+f_{h-2}\leq n<f_{h+1}+2f_{h-1}+1$;}
\put(78,0){Range V: $f_{h+1}+2f_{h-1}+1\leq n<3f_h+1$;}
\put(150,0){Range VI: $3f_h+1\leq n\leq 2f_{h+1}$.}
\end{picture}
\begin{picture}(205,131)
\put(0,125){$n\in\mathrm{I}$:}
\put(13,125){$R_a=[\underbrace{n-f_{h}-f_{h-3}-2\searrow}_{f_{h+1}+f_{h-1}-n},
\underbrace{2n-f_{h+1}-3f_{h-1}-2\longrightarrow}_{n-3f_{h-1}-1},
\underbrace{2n-f_{h+1}-3f_{h-1}-1\nearrow}_{f_{h+1}+f_{h-1}-n-1},
\underbrace{n-f_{h}-f_{h-3}-2\longrightarrow}_{n-2f_{h}+1},$}
\put(23.5,115){$\underbrace{n-f_{h}-f_{h-3}-1\nearrow}_{f_{h+1}+f_{h-1}-n+1},
\underbrace{f_{h}\longrightarrow}_{n-2f_{h}-1},
\underbrace{f_{h}-1\searrow}_{f_{h-4}+1},
\underbrace{f_{h-1}+f_{h-3}-2\longrightarrow}_{f_{h+1}+f_{h-1}-n-1},
\underbrace{f_{h-1}+f_{h-3}-1\nearrow}_{f_{h-4}+1},$}
\put(23.5,105){$\underbrace{f_{h}\longrightarrow}_{n-2f_{h}-1},
\underbrace{f_{h}-1\searrow}_{f_{h+1}+f_{h-1}-n+2},
\underbrace{n-f_{h}-f_{h-3}-2\longrightarrow}_{n-2f_{h}-1}]$;}
\put(13,95){$R_b=[\underbrace{n-f_{h}-f_{h-3}-2\searrow}_{f_{h+1}+f_{h-1}-n},
\underbrace{2n-f_{h+1}-3f_{h-1}-2\longrightarrow}_{n-3f_{h-1}-1},
\underbrace{2n-f_{h+1}-3f_{h-1}-1\nearrow}_{f_{h+1}+f_{h-1}-n-1},
\underbrace{n-f_{h}-f_{h-3}-2\longrightarrow}_{n-f_{h+1}+2}]$.}
\put(0,85){$n\in\mathrm{II}$:}
\put(13,85){$R_a=[\underbrace{f_{h}-2\longrightarrow}_{f_{h-1}-1},f_{h}-1,
\underbrace{f_{h}\longrightarrow}_{f_{h-3}-1},
\underbrace{f_{h}-1\searrow}_{f_{h-4}-1},
\underbrace{f_{h}-f_{h-4}\longrightarrow}_{3},
\underbrace{f_{h}-f_{h-4}+1\nearrow}_{f_{h-4}-1},
\underbrace{f_{h}\longrightarrow}_{f_{h-3}-1},
f_{h}-1,\underbrace{f_{h}-2\longrightarrow}_{f_{h-3}}$];}
\put(0,75){$n\in\mathrm{III}$:}
\put(13,75){$R_a=[\underbrace{n-f_{h}-f_{h-3}-2\longrightarrow}_{f_{h+1}+2f_{h-1}-n},
\underbrace{n-f_{h}-f_{h-3}-2\searrow}_{n-f_{h+1}-f_{h-1}-1},
\underbrace{f_{h}\longrightarrow}_{f_{h-3}},
\underbrace{f_{h}-1\searrow}_{2f_{h}+f_{h-2}-n-1},
\underbrace{n-f_{h}-f_{h-2}\longrightarrow}_{n-f_{h+1}-f_{h-1}+3},
\underbrace{n-f_{h}-f_{h-2}+1\nearrow}_{2f_{h}+f_{h-2}-n-1},$}
\put(23.5,65){$\underbrace{f_{h}\longrightarrow}_{f_{h-3}},
\underbrace{f_{h}\nearrow}_{n-f_{h+1}-f_{h-1}-1},
\underbrace{n-f_{h}-f_{h-3}-2\longrightarrow}_{f_{h-3}+1}]$;}
\put(120,65){$n\in\mathrm{II}\cup\mathrm{III}\cup\mathrm{IV}$:}
\put(150,65){$R_b=[\underbrace{n-f_{h}-f_{h-3}-2\longrightarrow}_{f_{h}}]$.}
\put(0,55){$n\in\mathrm{IV}$:}
\put(13,55){$R_a=[\underbrace{n-f_{h}-f_{h-3}-2\longrightarrow}_{f_{h+1}+2f_{h-1}-n},
\underbrace{n-f_{h}-f_{h-3}-2\searrow}_{n-f_{h+1}-f_{h-1}-1},
\underbrace{f_{h}\longrightarrow}_{3f_{h}-n+1},
\underbrace{f_{h}\nearrow}_{n-f_{h+1}-f_{h-1}-1},
\underbrace{n-f_{h}-f_{h-3}-2\longrightarrow}_{f_{h-3}+1}]$;}
%\put(170,55){$R_b=[\underbrace{n-f_{h}-f_{h-3}-2\longrightarrow}_{f_{h}}]$;}
\put(0,45){$n\in \mathrm{V}$:}
\put(13,45){$R_a=[\underbrace{f_{h+1}-2\searrow}_{f_{h-1}-1},
\underbrace{f_{h}\longrightarrow}_{f_{h+2}+f_{h-2}-n+1},
\underbrace{f_{h}\nearrow}_{n-f_{h+1}-f_{h-1}-2},
\underbrace{n-f_{h}-f_{h-3}-2\longrightarrow}_{n-f_{h+3}+f_{h-4}-3},
\underbrace{n-f_{h}-f_{h-3}-3\searrow}_{n-f_{h+2}-f_{h-3}-1}]$;}
\put(13,35){$R_b=[\underbrace{f_{h+1}-2\searrow}_{f_{h-1}-1},
\underbrace{f_{h}\longrightarrow}_{f_{h+2}+f_{h-2}-n+1},
\underbrace{f_{h}\nearrow}_{f_{h-1}-2},
\underbrace{f_{h+1}-2\longrightarrow}_{n-f_{h+2}+2}$.}
\put(0,25){$n\in \mathrm{VI}$:}
\put(13,25){$R_a=[\underbrace{f_{h+1}-2\searrow}_{f_{h+2}+f_{h}-n},
\underbrace{n-2f_{h}-2\longrightarrow}_{n-f_{h+2}-f_{h-2}-1},
\underbrace{n-2f_{h}-1\nearrow}_{2f_{h-2}-1},
\underbrace{n-f_{h}-f_{h-3}-2\longrightarrow}_{f_{h+2}+f_{h-1}-n+3},
\underbrace{n-f_{h}-f_{h-3}-3\searrow}_{n-f_{h+2}-f_{h-2}-1},
\underbrace{f_{h+1}+f_{h-4}-2\longrightarrow}_{f_{h+2}+f_{h}-n-1},$}
\put(23.5,15){$\underbrace{f_{h+1}+f_{h-4}-1\nearrow}_{n-f_{h+2}-f_{h-2}-1},
\underbrace{n-f_{h}-f_{h-3}-2\longrightarrow}_{f_{h+2}+f_{h-1}-n+3},
\underbrace{n-f_{h}-f_{h-3}-3\searrow}_{n-f_{h+2}-f_{h-3}-1}$,}
\put(13,5){$R_b=[\underbrace{f_{h+1}-2\searrow}_{f_{h+2}+f_{h}-n},
\underbrace{n-2f_{h}-2\longrightarrow}_{n-f_{h+2}-f_{h-2}-1},
\underbrace{n-2f_{h}-1\nearrow}_{f_{h+2}+f_{h}-n-1},
\underbrace{f_{h+1}-2\longrightarrow}_{n-f_{h+2}+2}]$.}
\end{picture}
\caption{The expressions of $R_a$ and $R_b$ for $r=2+\epsilon$.
$|R_a|=f_{h+1}$ and $|R_b|=f_{h}$ for $n\in \mathrm{I}\cup \mathrm{II}\cup \mathrm{III}\cup \mathrm{IV}$;
$|R_a|=f_{h+2}$ and $|R_b|=f_{h+1}$ otherwise.\label{Fig:2}}
\end{figure}

\begin{figure}[!ht]
\centering
\scriptsize
\setlength{\unitlength}{0.82mm}
\begin{picture}(205,30)
\put(0,25){Divide $3f_h\leq n<3f_{h+1}$ into ten ranges:}
\put(0,20){Range I: $3f_h\leq n<f_{h+2}+2f_{h-3}+1$;}
\put(100,20){Range II: $f_{h+2}+2f_{h-3}+1\leq n<f_{h+2}+f_{h-1}$;}
\put(0,15){Range III: $f_{h+2}+f_{h-1}\leq n<f_{h+1}+3f_{h-1}-1$;}
\put(100,15){Range IV: $f_{h+1}+3f_{h-1}-1\leq n<f_{h+2}+f_{h}-1$;}
\put(0,10){Range V: $f_{h+2}+f_{h}-1\leq n<4f_{h}-f_{h-4}$;}
\put(100,10){Range VI: $4f_{h}-f_{h-4}\leq n<f_{h+1}+4f_{h-1}$;}
\put(0,5){Range VII: $f_{h+1}+4f_{h-1}\leq n<f_{h+3}-1$;}
\put(100,5){Range VIII: $f_{h+3}-1\leq n<f_{h+2}+3f_{h-1}-1$;}
\put(0,0){Range IX: $f_{h+2}+3f_{h-1}-1\leq n<f_{h+1}+3f_{h}+1$;}
\put(100,0){Range X: $f_{h+1}+3f_{h}+1\leq n<3f_{h+1}$.}
\end{picture}
\begin{picture}(205,175)
\put(0,170){$n\in\mathrm{I}\cup\mathrm{II}$:}
\put(20,170){$R_a=[\underbrace{n-f_{h+2}\nearrow}_{f_{h+2}+f_{h-1}-n-1}
\underbrace{f_{h-1}-2\longrightarrow}_{n-f_{h+1}-f_{h-1}+1}
\underbrace{f_{h-1}-2\searrow}_{f_{h-2}}
\underbrace{f_{h-3}-1\longrightarrow}_{f_{h+2}+f_{h-1}-n}
\underbrace{f_{h-3}\nearrow}_{n-3f_{h}+1}
\underbrace{n-3f_{h}+f_{h-3}\longrightarrow}_{f_{h+2}+f_{h}-n-2}$}
\put(33.5,160){$\underbrace{n-3f_{h}+f_{h-3}-1\searrow}_{n-3f_{h}+1}
\underbrace{f_{h-3}-1\longrightarrow}_{f_{h+2}+f_{h-1}-n}
\underbrace{f_{h-3}\nearrow}_{n-f_{h+1}-2f_{h-1}}]$;}
\put(0,150){$n\in\mathrm{I}$:}
\put(13,150){$R_b=[\underbrace{n-f_{h+2}\nearrow}_{f_{h+2}+f_{h-1}-n-1}
\underbrace{f_{h-1}-2\longrightarrow}_{n-f_{h+1}-f_{h-1}+2}
\underbrace{f_{h-1}-3\searrow}_{f_{h-2}-2}
\underbrace{f_{h-3}-1\longrightarrow}_{5f_{h-1}-n+1}
\underbrace{f_{h-3}\nearrow}_{n-3f_{h}+f_{h-4}}]$;}
\put(0,140){$n\in\mathrm{II}$:}
\put(13,140){$R_b=[\underbrace{n-f_{h+2}\nearrow}_{f_{h+2}+f_{h-1}-n-1}
\underbrace{f_{h-1}-2\longrightarrow}_{n-f_{h+1}-f_{h-1}+2}
\underbrace{f_{h-1}-3\searrow}_{f_{h+1}+3f_{h-1}-n-1}
\underbrace{n-f_{h+1}-2f_{h-1}-1\longrightarrow}_{n-5f_{h-1}-1}
\underbrace{n-f_{h+1}-2f_{h-1}-1\nearrow}_{f_{h-3}+1}]$;}

\put(0,130){$n\in\mathrm{III}\cup\mathrm{IV}$:}
\put(20,130){$R_a=[\underbrace{f_{h-1}-2\longrightarrow}_{f_{h}}
\underbrace{f_{h-1}-2\searrow}_{f_{h+2}+f_{h}-n}
\underbrace{n-3f_{h}-1\longrightarrow}_{n-f_{h+2}-f_{h-1}}
\underbrace{n-3f_{h}\nearrow}_{f_{h-3}+1}
\underbrace{n-2f_{h}-2f_{h-2}\longrightarrow}_{f_{h+2}+f_{h}-n-2}
\underbrace{n-2f_{h}-2f_{h-2}-1\searrow}_{f_{h-3}+1}$}
\put(30.5,120){$\underbrace{n-3f_{h}-1\longrightarrow}_{n-f_{h+2}-f_{h-1}}
\underbrace{n-3f_{h}\nearrow}_{f_{h+2}+f_{h}-n-2}
\underbrace{f_{h-1}-2\longrightarrow}_{n-f_{h+2}-f_{h-1}+2}]$;}
\put(0,110){$n\in\mathrm{III}$:}
\put(13,110){$R_b=[\underbrace{f_{h-1}-2\longrightarrow}_{f_{h}+1}
\underbrace{f_{h-1}-3\searrow}_{f_{h+1}+3f_{h-1}-n-1}
\underbrace{n-f_{h+1}-2f_{h-1}-1\longrightarrow}_{n-5f_{h-1}-1}
\underbrace{n-f_{h+1}-2f_{h-1}-1\nearrow}_{f_{h+1}+3f_{h-1}-n-1}
\underbrace{f_{h-1}-2\longrightarrow}_{n-f_{h+2}-f_{h-1}+2}]$;}
\put(0,100){$n\in\mathrm{V}$:}
\put(13,100){$R_a=[\underbrace{f_{h-1}-2\longrightarrow}_{f_{h+2}+2f_{h}-n-1}
\underbrace{f_{h-1}-2\nearrow}_{n-f_{h+2}-f_{h}+1}
\underbrace{n-3f_{h}-1\longrightarrow}_{f_{h-2}}
\underbrace{n-3f_{h}\nearrow}_{2f_{h+1}+f_{h-1}-n-1}
\underbrace{f_{h-1}+f_{h-3}-2\longrightarrow}_{n-f_{h+2}-f_{h}+2}
\underbrace{f_{h-1}+f_{h-3}-3\searrow}_{2f_{h+1}+f_{h-1}-n-1}$}
\put(23.5,90){$\underbrace{n-3f_{h}-1\longrightarrow}_{f_{h-2}-1}
\underbrace{n-3f_{h}-2\searrow}_{n-f_{h+2}-f_{h}+1}
\underbrace{f_{h-1}-2\longrightarrow}_{f_{h-2}}]$;}
\put(115,90){$n\in\mathrm{IV\cup V\cup\ldots\cup IX}$:
$R_b=[\underbrace{f_{h-1}-2\longrightarrow}_{f_{h+1}}]$;}
\put(0,80){$n\in\mathrm{VI}$:}
\put(13,80){$R_a=[\underbrace{f_{h-1}-2\longrightarrow}_{f_{h+2}+2f_{h}-n}
\underbrace{f_{h-1}-2\nearrow}_{n-f_{h+2}-f_{h}}
\underbrace{n-3f_{h}-1\longrightarrow}_{f_{h+3}-n-2}
\underbrace{n-3f_{h}-1\searrow}_{n-2f_{h+1}-f_{h-1}+2}
\underbrace{f_{h-1}+f_{h-3}-2\longrightarrow}_{f_{h+1}+4f_{h-1}-n}
\underbrace{f_{h-1}+f_{h-3}-1\nearrow}_{n-2f_{h+1}-f_{h-1}+1}$}
\put(23.5,70){$\underbrace{n-3f_{h}-1\longrightarrow}_{f_{h+3}-n-2}
\underbrace{n-3f_{h}-2\searrow}_{n-f_{h+2}-f_{h}+1}
\underbrace{f_{h-1}-2\longrightarrow}_{f_{h-2}}]$;}

\put(0,60){$n\in\mathrm{VII}$:}
\put(15,60){$R_a=[\underbrace{f_{h-1}-2\longrightarrow}_{f_{h+2}+2f_{h}-n}
\underbrace{f_{h-1}-1\nearrow}_{n-f_{h+2}-f_{h}}
\underbrace{n-3f_{h}-1\longrightarrow}_{f_{h+3}-n-2}
\underbrace{n-3f_{h}-1\searrow}_{f_{h-3}+2}
\underbrace{n-2f_{h+1}-2\longrightarrow}_{n-f_{h+1}-4f_{h-1}}
\underbrace{n-2f_{h+1}-1\nearrow}_{f_{h-3}+1}$}
\put(25.5,50){$\underbrace{n-3f_{h}-1\longrightarrow}_{f_{h+3}-n-2}
\underbrace{n-3f_{h}-2\searrow}_{n-f_{h+2}-f_{h}+1}
\underbrace{f_{h-1}-2\longrightarrow}_{f_{h-2}}]$;}
\put(0,40){$n\in\mathrm{VIII}$:}
\put(15,40){$R_a=[\underbrace{f_{h-1}-2\longrightarrow}_{f_{h+2}+2f_{h}-n}
\underbrace{f_{h-1}-1\nearrow}_{f_{h-1}-2}
\underbrace{2f_{h-1}-3\longrightarrow}_{n-f_{h+3}+2}
\underbrace{2f_{h-1}-3\searrow}_{f_{h+2}+3f_{h-1}-n-1}
\underbrace{n-f_{h+2}-f_{h-1}-2\longrightarrow}_{n-f_{h+1}-4f_{h-1}+1}$}
\put(25.5,30){$\underbrace{n-f_{h+2}-f_{h-2}-1\nearrow}_{f_{h+3}+f_{h-3}-n-2}
\underbrace{2f_{h-1}-3\longrightarrow}_{n-f_{h+3}+3}
\underbrace{2f_{h-1}-4\searrow}_{f_{h-1}-1}
\underbrace{f_{h-1}-2\longrightarrow}_{f_{h-2}}]$;}
\put(0,20){$n\in\mathrm{IX}$:}
\put(13,20){$R_a=[\underbrace{f_{h-1}-2\longrightarrow}_{f_{h+2}+2f_{h}-n}
\underbrace{f_{h-1}-1\nearrow}_{f_{h-1}-2}
\underbrace{2f_{h-1}-3\longrightarrow}_{n-f_{h+2}-f_{h}+3}
\underbrace{2f_{h-1}-4\searrow}_{f_{h-1}-1}
\underbrace{f_{h-1}-2\longrightarrow}_{f_{h-2}}]$;}
\put(0,10){$n\in\mathrm{X}$:}
\put(13,10){$R_a=[\underbrace{n-4f_{h}-1\nearrow}_{f_{h+3}+f_{h}-n-2}
\underbrace{2f_{h-1}-3\longrightarrow}_{n-f_{h+2}-f_{h}+3}
\underbrace{2f_{h-1}-4\searrow}_{f_{h-1}-2}
\underbrace{f_{h-1}-2\longrightarrow}_{f_{h+1}+5f_{h}-n+1}
\underbrace{f_{h-1}-1\nearrow}_{n-f_{h+1}-3f_{h}}]$;}
\put(13,0){$R_b=[\underbrace{n-4f_{h}-1\nearrow}_{f_{h+3}+f_{h}-n-2}
\underbrace{2f_{h-1}-3\longrightarrow}_{n-f_{h+2}-f_{h}+3}
\underbrace{2f_{h-1}-4\searrow}_{f_{h-1}-2}
\underbrace{f_{h-1}-2\longrightarrow}_{5f_{h}-n+1}
\underbrace{f_{h-1}-1\nearrow}_{n-f_{h+1}-3f_{h}}]$.}
\end{picture}
\vspace{0.2cm}
\caption{The expressions of $R_a$ and $R_b$ for $r=3$.
$|R_a|=f_{h+1}$ and $|R_b|=f_{h}$ for $n\not\in$ Range $\mathrm{X}$;
$|R_a|=f_{h+2}$ and $|R_b|=f_{h+1}$ otherwise.\label{Fig:3}}
\end{figure}

There exists some positions that do not start a cube.
In fact, $\F[1,n]$ is not a cube for all $n\geq1$.

\begin{property}[]\
For $m\geq0$ the number of positions starting a cube in $\F[1,f_{m}]$ is
\begin{equation*}
C(m)=\begin{cases}
f_{2k}-k^2-k-1,&m=2k,\\
f_{2k+1}-k^2-2k-2,&m=2k+1.
\end{cases}
\end{equation*}
\end{property}

\begin{proof} %We consider the number of positions starting a cube in $\F[1,f_m]$.
We only need to prove the relation below by induction
$$C(m)=C(m-1)+C(m-2)-C(m-3)+\#R_b^{3,m-2}.$$
Here $\#R_b^{3,m-2}=\#[1,2,\ldots,f_{m-3}-1,\underbrace{0,\ldots,0}_{f_{m-4}+1}]=f_{m-3}-1$.
We give this relation in Figure \ref{Fig:14} for $m=7$.
\begin{figure}[!t]
\centering
\scriptsize
\setlength{\unitlength}{0.82mm}
\begin{picture}(180,73)
\put(21.6,68){$i$}
\put(25,68){=}
\put(30,68){${\begin{array}{*{34}{p{0.3cm}}}
1&2&3&4&5&6&7&8&9&10&11&12&13&14&15&16&17&18&19&20&21&22&23&24&25&26&27&28&29&30&31&32&33& 34\end{array}}$}
\put(6.2,46){$\of^{2,2}[1,26]$}
\put(25,46){=}
\put(30,46){${\begin{array}{*{34}{p{0.3cm}}}
0&0&0&0&0&1&0&0&0&0&0&0&0&0&0&0&0&0&1&0&0&0&0&0&0&0&1&0&0&0&0&0&0&0\end{array}}$}
\put(6.2,40){$\of^{2,3}[1,26]$}
\put(25,40){=}
\put(30,40){${\begin{array}{*{34}{p{0.3cm}}}
0&0&0&0&0&0&0&0&1&2&0&0&0&0&0&0&0&0&0&0&0&0&0&0&0&0&0&0&0&1&2&0&0&0\end{array}}$}
\put(6.2,34){$\of^{2,4}[1,26]$}
\put(25,34){=}
\put(30,34){${\begin{array}{*{34}{p{0.3cm}}}
0&0&0&0&0&0&0&0&0&0&0&0&0&1&2&3&4&0&0&0&0&0&0&0&0&0&0&0&0&0&0&0&0&0\end{array}}$}
\put(6.2,28){$\of^{2,5}[1,26]$}
\put(25,28){=}
\put(30,28){${\begin{array}{*{34}{p{0.3cm}}}
0&0&0&0&0&0&0&0&0&0&0&0&0&0&0&0&0&0&0&0&0&1&2&3&4&5&6&7&0&0&0&0&0&0\end{array}}$}
\put(67,54){\normalsize{\textbf{C(6)}}}
\put(29,33){\line(1,0){87}}
%\put(29,57){\line(1,0){87}}
\put(29,33){\line(0,1){18}}
\put(116,33){\line(0,1){26}}
\put(29,50){$\overbrace{\hspace{7.1cm}}$}
\put(116,39){\line(1,0){54}}
%\put(116,57){\line(1,0){54}}
\put(149,45){\line(0,1){5}}
\put(117,50){$\overbrace{\hspace{2.6cm}}$}
\put(129,54){\normalsize{\textbf{C(4)}}}
\put(170,39){\line(0,1){20}}
\put(116,45){\line(1,0){33}}
\put(117,58){$\overbrace{\hspace{4.3cm}}$}
\put(139,62){\normalsize{\textbf{C(5)}}}
\put(117,28){$\underbrace{\hspace{4.3cm}}$}
\normalsize
\put(139,20){$\mathbf{R_b^{3,5}}$}
\end{picture}
\vspace{-1.8cm}
\caption{The relation among functions $C(m)$ for $m=4,5,6,7$.\label{Fig:14}}
\end{figure}
\end{proof}

\noindent\textbf{Remark.} Since $\lim_{m\rightarrow\infty}\frac{C(m)}{f_m}=1$, %the asymptotic solution of $C(m)$ is
almost all positions in $\F$ start some cubes.
Let us consider factor spectrum
$\mathcal{S}_{\PP_5}=\{(\w,j)\mid \F[j;3|\w|]=\w\w\w\}$.
The asymptotic solution is equivalent to that for almost all $j\geq1$ there exists a factor $\w$ such that $(\w,j)\in\mathcal{S}_{\PP_5}$.

%\section{The factor spectrums of some results}

\bigskip

\noindent\emph{Acknowledgements.}
%We thank sincerely the referees for their helpful suggestions.
The research is supported by the Grant NSFC No.11701024 and No.11431007.

%\end{CJK*}

\begin{thebibliography}{AB}

\bibitem{AS2003}  J. P. Allouche, J. Shallit. \emph{Automatic sequences: Theory, applications, generalizations}. Cambridge University Press, Cambridge, 2003.

\bibitem{B1996} J. Berstel. Recent results on Sturmian words. \emph{Salomaa Development in Language Theory World Scientific}. (1995) 13-24.

\bibitem{D1998} F. Durand. A characterization of substitutive sequences using return words, \emph{Discrete Math}. 179 (1998) 89-101.

\bibitem{DL2003} D. Damanik, D. Lenz, Powers in Sturmian sequences, \emph{European Journal of Combinatorics}. 24.4 (2003) 377-390.

\bibitem{DMSS2015-1} C.-F. Du, H. Mousavi, L. Schaeffer, J. Shallit. Decision Algorithms for Fibonacci-Automatic Words, with Applications to Pattern Avoidance. Presented at $15^e$ Journ$\acute{e}$es Montoises d'Informatique Th$\acute{e}$orique, 23-26 September 2014, Nancy, France.


\bibitem{DMSS2016-3} C.-F. Du, H. Mousavi, L. Schaeffer, J. Shallit.
Decision Algorithms for Fibonacci-Automatic Words, III: Enumeration and Abelian Properties.
\emph{Internat. J. Found. Comput. Sci.}  27.8 (2016) 943-963.

\bibitem{FS1998} A. S. Fraenkel, J. Simpson, How Many Squares Can a String Contain? \emph{Journal of Combinatorial Theory}. 82.1 (1998) 112-120.

\bibitem{FS1999} A. S. Fraenkel, J. Simpson, The exact number of squares in Fibonacci words, \emph{Theoretical Computer Science}. 218 (1999) 95-106.

\bibitem{FS2014} A. S. Fraenkel, J. Simpson, Corrigendum to ``The exact number of squares in Fibonacci words'', \emph{Theoretical Computer Science}. 547 (2014) 122.

\bibitem{G2006} A. Glen, \emph{On Sturmian and Episturmian Words, and Related Topics,} Ph.D. thesis. The University of Adelaide, Australia, 2006.


\bibitem{HKN2011} T. Harju, T. K$\ddot{a}$rki, D. Nowotka. The number of positions starting a square in binary words. \emph{Electronic Journal of Combinatorics}. 18.1 (2011), 106-108.

\bibitem{HW2015-1} Y.-K. Huang, Z.-Y. Wen. The sequence of return words of the Fibonacci sequence, \emph{Theoretical Computer Science}. 593 (2015) 106-116.

\bibitem{HW2015-2} Y.-K. Huang, Z.-Y. Wen, Kernel words and gap sequence of the Tribonacci sequence, \emph{Acta Mathematica Scientia (Series B)}. 36.1 (2016) 173-194.

\bibitem{HW2016-3} Y.-K. Huang, Z.-Y. Wen, The numbers of repeated palindromes in the Fibonacci and Tribonacci sequences. \emph{Discrete Applied Mathematics}. 230 (2017), 78-90.


\bibitem{IMS1997} C. S. Iliopoulos, D. Moore, and W. F. Smyth. A characterization of the squares in a Fibonacci string. \emph{Theoretical Computer Science}. 172 (1997), 281-291.

\bibitem{L1983} M. Lothaire. \emph{Combinatorics on words, in: Encyclopedia of Mathematics and its applications}. Vol.17, Addison-Wesley, Reading, MA, 1983.

\bibitem{L2002} M. Lothaire. \emph{Algebraic combinatorics on words}, Cambridge Univ. Press, Cambridge, 2002.

\bibitem{MS2014} H. Mousavi, J. Shallit, Mechanical proofs of properties of the Tribonacci word, in \emph{Combinatorics on Words}. Springer International Publishing. (2014) 170-190.

\bibitem{S2010} K. Saari, Everywhere $\alpha$-repetitive sequences and Sturmian words, \emph{European Journal of Combinatorics}. 31.1 (2010) 177-192.

\bibitem{TW2007} B. Tan, Z.-Y. Wen, Some properties of the Tribonacci sequence.
\emph{European Journal of Combinatorics}. 28 (2007) 1703-1719.

\bibitem{WW1994} Z.-X. Wen, Z.-Y. Wen. Some properties of the singular words of the Fibonacci word, \emph{European Journal of Combinatorics}. 15 (1994) 587-598.

\bibitem{Z1972} E. Zeckendorf. Repr$\acute{e}$sentation des nombres naturels par une somme de nombres de Fibonacci ou de nombres de Lucas. \emph{Bulletin De La Societe Royale Des Sciences De Li$\grave{e}$ge}, 41(1972) 179-182.

\end{thebibliography}
\end{document}